\documentclass[letterpaper, 11pt,  reqno]{amsart}

\usepackage{amsmath,amssymb,amscd,amsthm,amsxtra, esint, enumerate}
\usepackage{mathtools} 
\usepackage{skak, bbm} 
\usepackage{ulem}
\usepackage[implicit=true]{hyperref}

\usepackage{mparhack}

\allowdisplaybreaks[2]

\usepackage{color}

\definecolor{gr}{rgb}   {0.,   0.69,   0.23 }
\definecolor{bl}{rgb}   {0.,   0.5,   1. }
\definecolor{mg}{rgb}   {0.85,  0.,    0.85}
\definecolor{yl}{rgb}   {0.8,  0.7,   0.}

\newcommand{\U}{\mathcal{U}}

\newtheorem{theorem}{Theorem} [section]

\newtheorem{lemma}[theorem]{Lemma}
\newtheorem{proposition}[theorem]{Proposition}
\newtheorem{remark}[theorem]{Remark}

\newtheorem{corollary}[theorem]{Corollary}

\newtheorem*{ackno}{Acknowledgments}

\DeclareMathOperator*{\intt}{\int}

\newcommand{\noi}{\noindent}
\newcommand{\Z}{\mathbb{Z}}
\newcommand{\R}{\mathbb{R}}

\newcommand{\T}{\mathbb{T}}

\let\Re=\undefined\DeclareMathOperator*{\Re}{Re}
\let\Im=\undefined\DeclareMathOperator*{\Im}{Im}

\let\P= \undefined
\newcommand{\P}{\mathbf{P}}

\newcommand{\Ha}{\mathbb{H}_a}

\newcommand{\Dr}{\Theta}
\newcommand{\dr}{\theta}
\newcommand{\E}{\mathbb{E}}

\DeclareMathOperator{\Law}{Law}

\newcommand{\NN}{\mathbb{N}}

\newcommand{\N}{\mathcal{N}}
\newcommand{\NB}{\mathbb{N}}

\newcommand{\F}{\mathcal{F}}

\newcommand{\al}{\alpha}

\newcommand{\dl}{\delta}

\newcommand{\nb}{\nabla}
\newcommand{\PP}{\mathbb{P}}

\newcommand{\eps}{\varepsilon}

\newcommand{\g}{\gamma}

\newcommand{\ld}{\lambda}

\newcommand{\s}{\sigma}

\newcommand{\ft}{\widehat}

\newcommand{\wt}{\widetilde}
\newcommand{\cj}{\overline}
\newcommand{\dx}{\partial_x}
\newcommand{\dt}{\partial_t}

\newcommand{\embeds}{\hookrightarrow}

\newcommand{\ta}{\theta}

\newcommand{\vp}{\varphi}

\renewcommand{\l}{\ell}
\renewcommand{\o}{\omega}

\newcommand{\les}{\lesssim}
\newcommand{\ges}{\gtrsim}

\newcommand{\wto}{\rightharpoonup}

\newcommand{\jb}[1]
{\langle #1 \rangle}

\newcommand{\W}{\mathcal{W}}

\numberwithin{equation}{section}
\numberwithin{theorem}{section}

\newcommand{\too}{\longrightarrow}

\usepackage{marginnote}

\begin{document}
\baselineskip = 12.7pt

\title[Quasi-invariance for FNLS]
{Transport of Gaussian measures under the flow of one-dimensional fractional nonlinear Schr\"{o}dinger equations
}

\author[J.~Forlano and K.~Seong]
{Justin Forlano and Kihoon Seong}

\address{
Justin Forlano, Maxwell Institute for Mathematical Sciences\\
 Department of Mathematics\\
 Heriot-Watt University\\
 Edinburgh\\ 
 EH14 4AS\\
  United Kingdom 
and Department of Mathematics\\
 University of California\\
  Los Angeles\\
   CA 90095\\
    USA}

\email{forlano@math.ucla.edu}

\address{
	Kihoon Seong\\
	Department of Mathematical Sciences\\
	Korea Advanced Institute of Science
	and Technology\\ 
	291 Daehak-ro, Yuseong-gu, Daejeon 34141, Republic of Korea
}

\email{hun1022kr@kaist.ac.kr}



\begin{abstract}
We study the transport
property of Gaussian measures on Sobolev spaces of periodic functions under the dynamics of the one-dimensional cubic fractional nonlinear Schr\"odinger equation. For the case of second-order dispersion or greater, 
we establish an optimal regularity result for the quasi-invariance of these Gaussian measures, following the approach by Debussche and Tsutsumi~\cite{DT2020}.
Moreover, we obtain an explicit formula for the Radon-Nikodym derivative and, as a corollary,
a formula for the two-point function arising in wave turbulence theory. We also obtain improved regularity results in the weakly dispersive case, extending those by the first author and Trenberth~\cite{FT}.
Our proof combines the approach introduced by Planchon, Tzvetkov and Visciglia~\cite{PTV} and that of Debussche and Tsutsumi~\cite{DT2020}.
\end{abstract}



%
\maketitle
%

\tableofcontents

\section{Introduction}

\subsection{Main results}
In this paper, we study the transport properties of Gaussian measures
on periodic functions under the dynamics of the cubic fractional nonlinear Schr\"odinger equation (FNLS)
on the one-dimensional torus $\T=\R/(2\pi\Z)$:
\begin{equation}\label{FNLS}
\begin{cases}
i\partial_t u+(-\dx^2)^\alpha u=\pm |u|^2u,\\
u|_{t=0}=\phi,
\end{cases}
\end{equation}
where $u:\R \times \T \longmapsto \mathbb{C}$ is the unknown function, and for $\al>0$, we denote by $(-\dx^{2})^{\al}$ the Fourier multiplier operator defined by $ ((-\dx^{2})^{\al}f )\,\,\widehat{}\,\,(n):=|n|^{2\al}\ft f(n)$, $n\in \Z$. Here, $\al$ measures the strength of the dispersion in FNLS~\eqref{FNLS}.
When $\al=1$, FNLS~\eqref{FNLS} is the cubic nonlinear Schr\"{o}dinger equation (NLS)
which arises in the study of nonlinear optics, fluids and plasma physics; for further discussion, see~\cite{Sulem}. 
For $\al=2$, \eqref{FNLS} corresponds to the cubic fourth order NLS (4NLS) and has applications in the study of solitons in magnetic materials~\cite{4NLS1, 4NLS2}. In the weakly dispersive case $\tfrac{1}{2}< \al<1$, FNLS~\eqref{FNLS} was introduced in the context of the fractional quantum mechanics~\cite{Laskin} and in continuum limits of long-range lattice interactions~\cite{FRAC1}. 
The equation FNLS~\eqref{FNLS} is no longer dispersive when $\al \leq \tfrac{1}{2}$ although the cubic nonlinear half-wave equation ($\al=\tfrac 12$) has many physical applications ranging from wave turbulence~\cite{FRAC3, FRAC2} to gravitational collapse~\cite{FRAC4, FRAC5}. 
See also \cite{Thirouin} regarding FNLS~\eqref{FNLS} when $\frac{1}{3}<\al <\frac{1}{2}$.
In this paper, we focus on studying FNLS~\eqref{FNLS} for $\al>\tfrac 12$ where dispersion is present.

Our goal in this paper is to study the statistical properties of solutions to FNLS~\eqref{FNLS}.
In particular, we study the quasi-invariance of Gaussian measures $\mu_s$ under the flow of \eqref{FNLS} where, for $s\in \R$, the measures $\mu_s$ are formally written as
\begin{align}
d\mu_s=Z_s^{-1}e^{-\frac{1}{2}\| \phi \|_{H^s}^2}\,d\phi=\prod_{n\in \mathbb{Z}}Z_{s,n}^{-1}e^{-\frac{1}{2}\jb{ n }^{2s} | \ft \phi_n |^2} \,d\ft \phi_n,
\label{gauss1}
\end{align}
where $Z_{s}$ and $Z_{s,n}$ are normalization constants.
These measures \eqref{gauss1} are
the induced probability measure under the random Fourier series:\footnote{In the following, we often drop the harmless factor of $2\pi$.}
\begin{align}
\omega \in \Omega \longmapsto \phi^{\omega}(x)=\phi(x;\omega)=\sum_{n\in \Z}
\frac{g_n(\omega)}{\jb{n}^s}e^{inx},
\label{series1}
\end{align}

\noi
where $\jb{\,\cdot\,} =(1+| \,\cdot \,|^2)^{\frac{1}{2}}$ and $\{g_n \}_{n\in \mathbb{Z}}$ is a sequence of independent standard complex-valued Gaussian random variables\footnote{By convention, we set  Var$(g_n)=1$, $n \in \Z$.} on a probability space $(\Omega,\mathcal{F}, \PP)$. 
From the random Fourier series representation, 
the random distribution $\phi^\o$ in \eqref{series1} lies in $H^{\s}(\T)$ 
almost surely
if and only if 
\begin{align}
\s < s -\frac 12.
\label{reg}
\end{align}
Consequently, $\mu_s$ is supported in $H^{\s}(\T)\setminus H^{s-\frac{1}{2}}(\T)$ where $\s$ satisfies \eqref{reg}. 
In particular, the triplet $(H^s, H^\s, \mu_s)$ forms an abstract Wiener space; see~\cite{GROSS, Kuo2}.
 In \cite{FT}, with Trenberth, the first author studied the transport property of Gaussian measures $\mu_s$ under the flow of FNLS~\eqref{FNLS} for some range of $s>\frac{1}{2}$, depending on $\al>\frac{1}{2}$. Our main goal is to improve the regularity restrictions in \cite{FT}.

 In order to discuss the transport of Gaussian measures $\mu_s$ under the flow of \eqref{FNLS}, we need to understand when there is a well-defined flow for \eqref{FNLS} in the support of $\mu_s$.
The well-posedness theory of \eqref{FNLS} is distinguished by the strength of the dispersion $\al$: (i) strong dispersion $\al\geq 1$ and (ii) weak dispersion $\frac{1}{2}<\al<1$. 
In the former case (i),
 for any $\al\ge 1$, the cubic FNLS~\eqref{FNLS} is globally well-posed in $H^{\s}(\T)$ for $\s \geq 0$~\cite{Bou, OTz, FT}. 
This global well-posedness result is sharp for the cubic NLS~\eqref{FNLS} ($\al=1$) and 4NLS ($\al=2$) as they are ill-posed in negative Sobolev spaces in the sense of non-existence of solutions~\cite{GO}; see also~\cite{BGT, Molinet, ChofPoc, OW1, OTz,  Oh, OW, Kishimoto}.    
In the latter case (ii), the cubic FNLS~\eqref{FNLS} is locally well-posed in $H^{\s}(\T)$ for $\s\geq \frac{1-\al}{2}$~\cite{Cho, ST1} and globally well-posed when $\s>\frac{10\al+1}{12}$~\cite{demirbas2013existence}. This latter result is not expected to be sharp.
For future use, we define $\Phi_t(\cdot): \phi \in H^{\s}(\T) \mapsto u(t,\phi)\in H^{\s}(\T)$ to be the flow map of FNLS~\eqref{FNLS} at time $t$.


In view of this well-posedness theory,  we consider Gaussian measures $\mu_s$ for
\begin{align}
s>\max\bigg( \frac{1}{2}, 1-\frac{\al}{2}\bigg). \label{optimals}
\end{align}
Our first main result is an optimal regularity result for the quasi-invariance of Gaussian measures under the flow of FNLS~\eqref{FNLS} in the strongly dispersive case $\al\geq 1$. 
This extends the result in \cite{FT} to the optimal Sobolev regularity.

\begin{theorem}
\label{THM:1}
Let $\al \ge 1$ and $s>\frac 12$. Then, the Gaussian measure $\mu_s$ in \eqref{gauss1} is quasi-invariant under the dynamics of the cubic FNLS \eqref{FNLS}. More precisely, 
for every $t\in \R$, the measures $(\Phi_t)_*\mu_s$ and $\mu_s$
are mutually absolutely continuous. 
\end{theorem}

Our second main result is an improvement on the regularity restrictions in \cite{FT} for the weakly dispersive case $\frac{1}{2}<\al<1$.

\begin{theorem}
\label{THM:2}
Let $\frac 12 < \al <1$ and $s>\frac{3-2\al}{2}$. Then, for every $R>0$, there exists $T>0$ such that
for every measurable set
\begin{align*}
A \subset \{ u\in H^{s-\frac 12-\eps}(\T)\, : \, \|u\|_{H^{s-\frac{1}{2}-\eps}(\T)}<R \}
\end{align*} 
satisfying $\mu_{s}(A)=0$, where $\eps>0$ is sufficiently small, we have $(\Phi_{t})_{\ast}\mu_{s}(A)=0$ for every $t\in [-T,T]$.
\end{theorem}

Theorem \ref{THM:1} is an optimal regularity result for the quasi-invariance of $\mu_s$ under the dynamics of \eqref{FNLS} when $\al \geq 1$. In particular, this includes the case of the cubic NLS ($\al=1$);  see Remark~\ref{Remark: Bou, Zhis quasi result}.
Our result improves on the main result in \cite{FT}, where the first author and Trenberth proved quasi-invariance of $\mu_s$ under \eqref{FNLS} for $s>\max(\frac 23, \frac {11}6-\al)$ when $\al \ge 1$. 
In addition, we also have an explicit formula for the corresponding Radon-Nikodym derivative and show that it is locally bounded with respect to $\mu_s$; see Proposition~\ref{PROP:Linfty}.
In the weakly dispersive setting, Theorem~\ref{THM:2} improves the regularity restriction of $s>\min\big( 1, \frac{11}{6}-\al\big)$ in \cite{FT}. 
Whilst our regularity result does not cover the full range in \eqref{optimals}, our problem forces us to go beyond current techniques. In particular, the key novelty in our proof is to combine the approaches in \cite{PTV} and \cite{DT2020}. Further improving the regularity restrictions in the weakly dispersive case $\frac{1}{2}<\al<1$ remains an interesting open problem.

\begin{remark}\rm \label{RMK:global}
The local-in-time nature of Theorem~\ref{THM:2} is an artefact of the lack of an almost sure global well-posedness result on the support of $\mu_s$ for FNLS~\eqref{FNLS} when $\tfrac{1}{2}<\al<1$ and $s\leq \tfrac{10\al+7}{12}$. 
Any improvement in this direction would lead to a corresponding improvement in Theorem~\ref{THM:2}; namely, we could `upgrade' local-in-time quasi-invariance to (global-in-time) quasi-invariance.
\end{remark}

The transport properties of Gaussian measures 
have been well studied in probability theory, beginning with the seminal work of Cameron-Martin \cite{Cameron};
see, \cite{Cameron, Ramer, Cru1, Cru2} for example. In particular, Ramer~\cite{Ramer} studied the quasi-invariance of Gaussian measures under general nonlinear transformations.
In \cite{Tz}, Tzvetkov initiated the study of the
transport properties of Gaussian measures on functions\,/\,distributions under nonlinear Hamiltonian PDEs
and there has been significant progress in this direction
\cite{Tz, OTz,  OTz2, OST, OTT,  PTV, GOTW, FT, STX, DT2020, OS}.

In \cite{Tz, OTz,  OTz2, OST, GOTW, FT, STX, OS}, 
an indirect method has been
an effective tool in showing the quasi-invariance property of the Gaussian measures under Hamiltonian flows. Namely, rather than directly proving the quasi-invariance of the Gaussian measure, one instead proves the quasi-invariance of a weighted Gaussian measure which is absolutely continuous with respect to the reference Gaussian measure. This indirect strategy is composed of (i) the construction of the weighted Gaussian measure, where the weight arises from correction terms stemming from the nonlinearity,
and (ii) an efficient energy estimate (with smoothing) on the time derivative of a corresponding modified energy.

In contrast, the approach we use to prove Theorem~\ref{THM:1} and Theorem~\ref{THM:2} is based on exploiting an explicit formula for the Radon-Nikodym derivative of the transported measure with an $L^2$-cutoff. This approach was introduced by Debussche and Tsutsumi \cite{DT2020}. More precisely, one obtains an explicit expression for the Radon-Nikodym derivative of this transported measure under dynamics which have been truncated to finite dimensions (truncated FNLS \eqref{TFNLS}) and prove the uniform $L^p$-integrability, $p>1$, of the Radon-Nikodym derivative in this truncation. 
We then obtain an explicit representation for the Radon-Nikodym derivative of the transported measure under the flow of FNLS \eqref{FNLS}, from which Theorem $\ref{THM:1}$ follows; see Section \ref{SEC:proof main}. 

In comparison to the indirect method mentioned above, an advantage of this method is that we can exploit 
the dispersion and temporal oscillations through the use of space-time estimates, such as the $L^4_{t,x}$-Strichartz estimates on $\T$. 
Namely, we establish an estimate for the integral in time of the derivative of the $H^{s}$-energy functional. 
 As an example, for solutions to FNLS~\eqref{FNLS}, the energy estimate (with smoothing) is
\begin{align}
\begin{split}
\bigg|\Re \int_{-T}^{0}  \jb{ i|\Phi_{t}(v)|^{2}\Phi_{t}(v),D^{2s}\Phi_{t}(v)}_{L^2(\T)}dt \bigg|  
&\leq C( \|\Phi_{t}(v)\|_{X_{T}^{0,\frac{1}{2}+}}) \| \Phi_t (v)\|_{X_{T}^{s-\frac{1}{2}-\eps,\frac{1}{2}+}}^{\ta},
\end{split}
\label{stenergy1}
\end{align}
where $\ta \leq 4$,\footnote{One can actually allow for $4\leq \ta <4+\eps_0$ for some small $\eps_0>0$; for example see Lemma~\ref{LEM:un1}.}, $\eps>0$ sufficiently small, $T>0$ and the spaces $X_{T}^{s,b}$ are the local-in-time Fourier restriction norm spaces $X^{s,b}$ which are defined in Subsection~\ref{SUBSEC:spaces}.
This is in contrast to the approaches 
in \cite{Tz, OTz,  OTz2, OST, GOTW, FT, STX} where the main energy estimates are established only for each fixed time $t$.  However, these approaches do have a complementing advantage: by reducing the analysis to time $t=0$, one can exploit random oscillations
in the random Fourier series \eqref{series1} and obtain a probabilistic energy estimate. This approach seems to work well when the dispersion is weaker, such as for nonlinear wave equations; see \cite{OTz2, GOTW, STX}.
 
We now return to discussing our main results Theorem~\ref{THM:1} and Theorem~\ref{THM:2}.
Let us first consider the case of high dispersion $\al \geq 1$ (Theorem~\ref{THM:1}).
Following the analysis in \cite{DT2020}, we can only obtain the optimal regularity result $s>\frac{1}{2}$ up to third-order dispersion ($\al \ge \frac 32$). Hence, below third-order dispersion ($1\leq \al< \frac 32$), we need to refine the analysis in \cite{DT2020}. 
The essential difference between our analysis and that in~\cite{DT2020} is that we employ a symmetrization argument, as in \cite{HPT, FT}, which allows us to gain extra decay through the mean value theorem and the double mean value theorem (Lemma~\ref{DMVT}).

In the weakly dispersive case $\frac{1}{2}<\al<1$, the method in \cite{DT2020} is no longer appropriate since we must work with the density of the transported measure with an $L^{2}$-cutoff. This $L^2$-cutoff is necessary 
in order to obtain the uniform $L^p$-integrability of the corresponding Radon-Nikodym derivative. 
The disadvantage here is that we need deterministic control of the flow of individual solutions by the $L^2$-norm of the initial data. 
In view of the regularity gap between the known well-posedness of FNLS~\eqref{FNLS} and $L_{x}^2$, 
we expect such control only through the $L^2$-conservation of solutions to \eqref{FNLS} at
each fixed time. 
Hence, we must use the weaker space-time norm $L^{\infty}([0,T];L^2_{x})$ instead of the $X^{0,\frac{1}{2}+}_{T}$ norm in \eqref{stenergy1}.
Note that this gap is due to the weaker dispersion which causes a derivative loss in the $L^{4}_{t,x}$-Strichartz estimate; see Lemma~\ref{LEM:StrichWD}. 
Thus, following the approach in \cite{DT2020}, we can only ever hope to obtain results for some $\al>\al_0>\frac{1}{2}$ and with a more restrictive range of $s$ than stated in Theorem~\ref{THM:2}; see Remark~\ref{RMK:DTfrac} for more details.

To go beyond this difficulty, we combine the local argument in \cite{PTV} with the density based approach in \cite{DT2020}. The benefit of our approach is two-fold: (i) we are no longer constrained by an $L^2$-cutoff in the measure and (ii) we weaken the space-time energy estimate \eqref{stenergy1} to 
\begin{align}
\begin{split}
\bigg|\Re \int_{-T}^{0}  \jb{ i|\Phi_{t}(v)|^{2}\Phi_{t}(v),D^{2s}\Phi_{t}(v)}_{L^2(\T)}dt \bigg|  
&\leq C\big(1+\| \Phi_t (v)\|_{X_{T}^{s-\frac{1}{2}-\eps,\frac{1}{2}+}}^{k}\big),
\end{split}
\label{stenergy2}
\end{align}
for some $k\in \NB$.
This allows us to obtain Theorem~\ref{THM:2}. 
In proving \eqref{stenergy2} for the weakly dispersive FNLS~\eqref{FNLS}, we exploit both the linear and bilinear Strichartz estimates in~\cite{ST1}.

\begin{remark}\rm
We point out that the same issue of dealing with $L^2$-cutoffs on the probability measures was faced in the analysis of FNLS~\eqref{FNLS} for $\frac{1}{2}<\al<1$ in \cite{FT}. 
In order to prove a quasi-invariance result for the full range $\frac{1}{2}<\al<1$, the first author and Trenberth applied the method due to Planchon, Tzvetkov and Visciglia~\cite{PTV}. The key point here is to argue locally within $H^{\s}(\T)$, which allows one to use deterministic growth bounds on solutions to \eqref{FNLS} to weaken the necessary fixed-time energy estimate. In \cite{GOTW}, the authors combined the approach in \cite{PTV} with the energy method in \cite{Tz, OTz,  OTz2} to produce a hybrid argument which has a further weakened energy estimate. This hybrid argument was then applied in \cite{FT} to obtain an improved regularity restriction in the restricted range $\al>\frac{5}{6}$.
\end{remark}

\begin{remark}\label{Remark: Bou, Zhis quasi result} \rm
\textup{(i)} For the particular case of the cubic NLS ($\al=1$), Theorem~\ref{THM:1} implies that the Gaussian measures $\mu_s$ are quasi-invariant under this NLS flow for any $s>\frac{1}{2}$. This extends implicit results in the works of Bourgain~\cite{Bo94} and Zhidkov~\cite{Zhid}.
Those authors showed that for each $k\in \mathbb{N}$, NLS has an invariant weighted Gaussian measure $\rho_k$ which is mutually absolutely continuous with the Gaussian measure $\mu_k$. The invariance of the measures $\rho_{k}$ imply the quasi-invariance of the Gaussian measures $\mu_k$ for each $k\in \mathbb{N}$. See also \cite{OTintro} for further discussion. 
\medskip

\noi
\textup{(ii)}  In this remark, we only consider the defocusing FNLS~\eqref{FNLS}; that is, \eqref{FNLS} with the negative sign.
Due to the Hamiltonian structure of FNLS~\eqref{FNLS}, it is natural to study the transport property of the following weighted Gaussian measure (Gibbs measure)
\begin{align*}
d\rho_{\al} = Z^{-1}e^{-\frac{1}{4}\int |u|^{4}dx}d\mu_{\al}.
\end{align*}
In~\cite{Bo94}, Bourgain proved that $\rho_{1}$ is invariant under the dynamics of NLS. This extends to $\al>\frac{2}{3}$, in the sense that $\rho_{\al}$ is invariant under the dynamics of FNLS~\eqref{FNLS}; see \cite{Demirbas}. In the remaining range $\frac{1}{2}<\al \leq \frac{2}{3}$, one no longer has deterministic well-posedness in the support of $\mu_{\al}$. In \cite{ST2}, Sun and Tzvetkov proved almost sure global well-posedness of FNLS~\eqref{FNLS} with respect to $\mu_{\al}$ for any $\al>\frac{31-\sqrt{233}}{28}\approx 0.562$ and hence the corresponding invariance of $\rho_{\al}$; see also \cite{ST1}.
 These results also imply the quasi-invariance of $\mu_{\al}$. Thus, the quasi-invariance of $\mu_{\al}$ under the dynamics of FNLS~\eqref{FNLS} with weak dispersion $\frac{1}{2}<\al<1$ may persist in regularities below those where deterministic (local) well-posedness holds.
  \end{remark}

\begin{remark}\rm
Whilst Theorem~\ref{THM:1} is optimal with respect to \eqref{optimals}, it is possible to inquire about the quasi-invariance of Gaussian measures supported below $L^2(\T)$ under suitably renormalized FNLS dynamics.
Recently, the second author with Oh~\cite{OS} proved the quasi-invariance of Gaussian measures $\mu_s$ in negative Sobolev spaces under the flow of the (renormalized) 4NLS for any $s>\frac{3}{10}$. 
In~\cite{OTzW}, Oh, Tzvetkov and Wang established the invariance of the white-noise measure $\mu_0$ under the flow of the (renormalized) 4NLS.
These results seem to be extendable to the (renormalized) FNLS for some $\al>1$.
\end{remark} 

\begin{remark}\rm 
The transport property of Gaussian-measures under the flow of Hamiltonian PDEs is intimately related to the strength of dispersive effects. Exploring this idea in the positive sense motivates our study of FNLS~\eqref{FNLS}. In the negative sense, Oh, Sosoe and Tzvetkov~\cite{OST} showed that Gaussian measures $\mu_s$ are not quasi-invariant under the following dispersionless ODE ($\al=0$): $i\dt u =|u|^2 u$.
See \cite{STX} for a similar result for a dispersionless system. Due to the change of variables $u(t,x)\mapsto u(t,x-t)$, this implies that $\mu_s$ is not quasi-invariant under the flow of the transport equation
\begin{align*}
i\dt u+i\dx u=|u|^2 u.
\end{align*}
We therefore expect that $\mu_s$ is also not quasi-invariant under the flow of the half-wave equation ($\al=\frac{1}{2}$), and it would be of interest to prove or disprove this claim. 
\end{remark}

\subsection{The Radon-Nikodym derivative of the transported measure}\label{SUBSEC:Density}

Before discussing further consequences of our results, we state the explicit formula for the Radon-Nikodym derivative of the transported measure when $\al\geq 1$; see Proposition~\ref{PROP:den} below. To this end, we first introduce some notations. Given $R > 0$, we use $B_R$ to denote the ball of radius $R$ in $L^2(\T)$
centered at the origin.
 Define $E_N$  by
\begin{align*}
E_N=\pi_{\leq N}L^2(\T)=\text{span} \{e^{inx}:\vert n \vert \leq N\}
\end{align*}

\noi
and let $E_N^\perp$ be the orthogonal complement of $E_N$
in $L^2(\T)$, where $\pi_{\leq N}$ is the (sharp) Dirichlet projection to frequencies $\{ |n|\leq N\}$.
Given $s\in \R$, let $\mu_s$ be the Gaussian measure on $H^{s-\frac{1}{2}-\eps }(\T)$ defined in \eqref{gauss1}. Then, we can write $\mu_s$ as
\begin{align*}
\mu_s=\mu_{s,N}\otimes \mu_{s,N}^\perp,
\end{align*}

\noi
where $\mu_{s,N}$ and $\mu_{s,N}^\perp$ are the marginal distributions of $\mu_s$ restricted to $E_N$ and $E_N^\perp$, respectively. In other words, $\mu_{s,N}$ and $\mu_{s,N}^\perp$ are induced probability measures under the following random Fourier series:
\begin{align*}
\pi_{\le N} \phi&:\omega \in \Omega \longmapsto 
\pi_{\le N} \phi(x;\omega)=\sum\limits_{ \vert n \vert \leq N}\frac{g_n(\omega)}{\jb{n}^s}e^{inx},\\
\pi_{> N} \phi&:\omega \in \Omega \longmapsto 
\pi_{> N} \phi(x;\omega)=\sum\limits_{\vert n \vert >N}\frac{g_n(\omega
	)}{\jb{n}^s}e^{inx},
\end{align*}

\noi
respectively. Formally, we can write $\mu_{s,N}$ and $\mu_{s,N}^\perp$ as
\begin{align*}
d\mu_{s,N}=Z_{s,N}^{-1}e^{-\frac{1}{2} \| \pi_{\leq N}\phi \|_{H^s}^2   }d\phi_N 
\qquad \text{and} \qquad 
d\mu_{s,N}^\perp=\widehat{Z}_{s,N}^{-1}e^{-\frac{1}{2}\| \pi_{>N}\phi\|_{H^s }^2  }d\phi_N^\perp, 
\end{align*}

\noi
where $d\phi_N$ and  $d\phi_N^\perp$ are (formally) the products of the Lebesgue measures
on the Fourier coefficients:
\begin{align*}
d\phi_N = \prod_{|n| \le N} d \ft \phi(n) 
\qquad \text{and} \qquad 
d\phi_N^\perp = \prod_{|n| >  N} d \ft \phi(n) .
\end{align*}

\noi
Given $r>0$, we also define measures $\mu_{s,r}$ and $\mu_{s,N,r}$ by
\begin{align*}
d\mu_{s,r}=\chi_{\{ \| \phi \|_{L^2}\le r  \} }d\mu_s \qquad \text{and} \qquad d\mu_{s,N,r}=\chi_{\{ \| \phi_N\|_{L^2}\le r  \} }d\mu_{s,N},
\end{align*}
where $\chi=\chi_{A}$ denotes the characteristic function of the set $A$.
Finally, we introduce the following truncated version of FNLS~\eqref{FNLS}. Given $N\in \NB$, we consider the Cauchy problem 
\begin{equation}
\begin{cases}
i\partial_t u_N+(-\dx^2)^\alpha u_N=\pm \pi_{\leq N}\big( |u_N|^2u_N\big), \quad (t,x)\in \R\times \T,\\
u|_{t=0}=\pi_{\leq N}\phi=\phi_N.
\end{cases}
\label{TFNLS}
\end{equation} 
When $\al\geq 1$, the global well-posedness 
of \eqref{TFNLS}  (Lemma~\ref{LEM:Superapprox}) implies we may define $\Phi_{N,t}(\cdot)$ to be the flow map of the truncated dynamics~\eqref{TFNLS} at time $t$: $\phi_N \in  H^{\s}(\T)  \to \Phi_{N,t}(\phi)=:u_N(t,\phi_N) \in H^{\s}(\T) $.

We now describe the approach of Debussche and Tsutsumi \cite{DT2020} for proving the quasi-invariance of Gaussian measures as applied to FNLS~\eqref{FNLS}. The key result is the following:

\begin{proposition}\label{PROP:den}
Let $\al\geq 1$, $N \in \NN$, $t\in \R$ and $r>0$. Then, the Radon-Nikodym derivative of the transported measure $(\Phi_{N,t})_*\mu_{s,N,r}$ with respect to $\mu_{s,N,r}$ is given by  

\begin{align}
\begin{split}
\frac {d({\Phi_{N,t}})_*\mu_{s,N,r} }{d\mu_{s,N,r} }&=f_N(t,\cdot) \\
&=
\exp\bigg(\mp\int_0^t  \textup{Re} \,  \jb{i (|u_N|^2u_N)(-t', \cdot ), D^{2s}u_N(-t', \cdot ) }_{L^2(\T)} dt' \bigg).
\end{split}
\label{R1}
\end{align}

\noi
 Suppose, in addition, that $\{f_N(t,\cdot )\}_{N \in \NN}$ is uniformly bounded in $L^p(d\mu_{s,r})$, for some $1<p<\infty$. Then, the Radon-Nikodym derivative of the transported measure $(\Phi_t)_*\mu_{s,r} $ with respect to $\mu_{s,r}$ is given by  
\begin{align*}
\frac {d({\Phi_t})_*\mu_{s,r} }{d\mu_{s,r} }=f(t,\cdot)=
\exp\bigg(\mp\int_0^t   \textup{Re} \, \jb{i (|u|^2u)(-t',\cdot ), D^{2s}u(-t',\cdot)   }_{L^2(\T)}    dt' \bigg)
\end{align*}

\noi
and it is in $L^p(d\mu_{s,r})$. 
\end{proposition}

The arguments in \cite{DT2020} essentially establish Proposition~\ref{PROP:den} when $\al\geq \frac{3}{2}$, which we extend to $\al \geq 1$.
In Proposition \ref{PROP:den}, the crucial assumption is the uniform $L^p$-integrability of the Radon-Nikodym derivative $f_N(t,\cdot)$ in \eqref{R1}, corresponding to the truncated flow \eqref{TFNLS}. This then implies the explicit formula for the Radon-Nikodym derivative of the transported measure $(\Phi_t)_{*}\mu_{s,r}$ with respect to $\mu_{s,r}$ and hence the mutual absolute continuity of $(\Phi_t)_{*}\mu_{s,r}$ and $\mu_{s,r}$.  Theorem \ref{THM:1} follows from this mutual absolute continuity and a limiting argument; see Section \ref{SEC:proof main}.

The proof of the uniform $L^p$-integrability of the Radon-Nikodym derivative is split into a deterministic step and a probabilistic step. In the deterministic step, one proves a deterministic energy estimate for $\log f_N(t,\cdot)$; see Lemma \ref{LEM:energy}.   
The probabilistic step then verifies the uniform $L^p$-integrability of this bound for $f_{N}(t,\cdot)$.
In \cite{DT2020}, the proof of this latter step uses a dyadic pigeon hole argument following~\cite{Bo94}, while we use the variational approach due to Barashkov and Gubinelli~\cite{BG} (the Bou\'e-Dupuis variational formula).

The approach outlined above is not the only way to verify the key assumption in Proposition~\ref{PROP:den}.
Indeed, higher integrability of the Radon-Nikodym derivative is equivalent to a quantitative quasi-invariance statement.

\begin{proposition} \textup{(}\cite[Proposition 3.4]{GLT}, \cite[Proposition 3.5]{BTh}\textup{)}\label{PROP:quantqi}
Let $s \in \R$, $\s\in \R$ satisfying \eqref{reg}, $\rho$ be a probability measure supported on $H^{\s}(\T)$ and, for each $t\in \R$, a measurable map $\Psi_{t}: H^{\s}(\T) \to H^{\s}(\T)$. 
Suppose there exists $0\leq \dl<1$ and $C=C(t,\dl)>0$ such that 
\begin{align}
(\Psi_{t})_{\ast}\rho(A)\le C \{\rho(A)\}^{1-\dl},
\label{Yudo}
\end{align}
for any measurable set $A\subseteq H^{\s}(\T)$. By the Radon-Nikodym theorem, there exists a non-negative $f(t,\cdot) \in L^1(d\rho)$ such that $d(\Psi_{t})_{\ast} \rho=f(t,\cdot)d\rho$. Then, 

\noi
\textup{(i)} The property \eqref{Yudo} holds for some $0\leq \dl<1$ if and only if $f(t,\cdot)\in  L^1(d\rho)\cap  L_{w}^{p}(d\rho)$ with $p=\frac{1}{\dl}$. Namely, $f(t,\cdot)\in L^1(d\rho)$ and there exists $C'>0$ such that we have 
\begin{align*}
\rho(\{ \phi\, :\,  |f(t,\phi)| >\ld ) \leq C' \jb{\ld}^{-p}, \quad \text{for all} \quad \ld >0.
\end{align*}

\noi
\textup{(ii)} The property \eqref{Yudo} holds for $\dl=0$ if and only if $f(t,\cdot)\in L^{\infty}(d\rho)$.

\end{proposition}

The above proposition holds in general whenever there are two finite measures $\mu,\nu$ on some measure space for which $\mu \ll \nu$, so that the Radon-Nikodym derivative exists. We considered particular measures in the above statement that are more relevant for our current discussion. 
We remark that in \cite{Tz, OTz,OST, OTz2}, the bounds \eqref{Yudo} were established for a weighted Gaussian measure with an appropriate cutoff and thus one has $L^p$-bounds for the associated Radon-Nikodym derivatives.

Applying our new approach in this paper to the strongly dispersive FNLS~\eqref{FNLS}, we obtain the following result.
\begin{proposition}\label{PROP:Linfty} Let $\al \geq 1$ and $s>\frac{1}{2}$.
For every $t\in \R$, the Radon-Nikodym derivative of the transported measure $(\Phi_t)_*\mu_s$ with respect to $\mu_s$, which exists from Theorem~\ref{THM:1}, belongs to $ L^1(d\mu_s)\cap L^{\infty}_{\textup{loc}}(d\mu_s) $. Moreover, we have the following explicit formula
\begin{align}
\frac {d({\Phi_t})_*\mu_{s} }{d\mu_{s} }(\phi)=f(t,\phi)=
\exp\bigg(\mp\int_0^t  \textup{Re} \,\jb{i (|u|^2u)(-t',\phi ), D^{2s}u(-t',\phi)   }_{L^2(\T)}    dt' \bigg)
\label{density}
\end{align}
for all $\phi\in L^2(\T)$. 
\end{proposition}

The claim \eqref{density} follows from Theorem~\ref{THM:1} and Proposition~\ref{PROP:den};  see Section \ref{SEC:weak}.
The local boundedness follows from using Lemma~\ref{LEM:energy} in the argument of Section \ref{SEC:weak}, which implies that \eqref{Yudo} holds for any measurable set $A\subseteq B_{R} \subset L^2(\T)$, $R>0$, and for $\dl=0$, where $\rho=\mu_{s}$ and $\Psi_{t}$ is the flow of the strongly dispersive FNLS~\eqref{FNLS}.
This is a general feature of our argument and a similar partial quantitative estimate is obtained by the approach in \cite{PTV} and the hybrid argument in \cite{GOTW} (albeit with some $\dl>0$ and for appropriately weighted Gaussian measures).

We note that the higher local integrability in Proposition~\ref{PROP:Linfty} has the following application to the higher (local) integrability of solutions to the infinite-dimensional Liouville equation associated to $\mu_s$:
\footnote{We thank Nikolay Tzvetkov for suggesting this application to us.}

\begin{proposition}\label{COR:stability}
Let $\al\geq 1$, $s>\frac{1}{2}$, $\Phi_{t}$ denote the flow of FNLS~\eqref{FNLS} and $f(t,\cdot)$ denote the Radon-Nikodym derivative of $(\Phi_{t})_{\ast}\mu_s$ with respect to $\mu_s$ as in Proposition~\ref{PROP:Linfty}.
Given any $1\leq p<\infty$, let $g_1, g_2 \in L^1(d\mu_s)\cap L^p_{\textup{loc}}(d\mu_s)$. Then, for any $t\in \R$, the transports of the measures 
\begin{align*}
g_1(\phi)d\mu_s(\phi), \quad \text{and} \quad g_2(\phi)d\mu_s(\phi)
\end{align*}
by $\Phi_t$ are given by 
\begin{align*}
G_1(t,\phi)d\mu_{s}(\phi)\quad \text{and} \quad G_2(t,\phi)d\mu_s(\phi),
\end{align*}
resepectively, for suitable $G_1(t,\cdot),G_{2}(t,\cdot)\in L^1(d\mu_s)\cap L^p_{\textup{loc}}(d\mu_s)$.
Moreover, for any $R>0$, we have 
\begin{align}
\big\| \chi_{B_R} \big( G_{1}(t,\cdot)-G_{2}(t,\cdot) \big)\big\|_{L^p(d\mu_s)}\leq \big\| \chi_{B_R} f(t,\cdot) \big \|_{L^{\infty}(d\mu_s)}^{1-\frac{1}{p}} \| \chi_{B_{R}}( g_1-g_2)\|_{L^p(d\mu_s)},
\label{Lpstab}
\end{align}
when $p>1$, and 
\begin{align}
\| G_1(t,\cdot)-G_2(t,\cdot)\|_{L^1(d\mu_s)}=\|g_1-g_2\|_{L^1(d\mu_s)}. \label{L1stab}
\end{align}
\end{proposition}

We point out that if the Radon-Nikodym derivative $f(t,\cdot)$ was (globally) in $L^{\infty}(d\mu_s)$, then \eqref{Lpstab} would become
\begin{align*}
\big\|  G_{1}(t,\cdot)-G_{2}(t,\cdot) \big\|_{L^p(d\mu_s)}\leq \big\| f(t,\cdot) \big \|_{L^{\infty}(d\mu_s)}^{1-\frac{1}{p}} \| g_1-g_2\|_{L^p(d\mu_s)}.
\end{align*}
The application of quasi-invariance of Gaussian measures (more generally, of arbitrary probability measures) to the $L^1$-stability statement, as in \eqref{L1stab}, was first noted by Sun and Tzvetkov in \cite{ST1}. 
The result of Proposition~\ref{COR:stability} shows that the higher (local) integrability of the Radon-Nikodym derivative with respect to the reference Gaussian measure $\mu_s$ implies stability locally in $L^p(d\mu_s)$, for $p>1$.

We also mention another application of Proposition~\ref{PROP:Linfty} to wave turbulence theory.
The main concern here is to 
obtain a statistical description of the out-of-equilibrium dynamics of a nonlinear
dispersive PDE; see for instance the books \cite{Zak1, Naz}.
In particular, by sampling initial data from some probability distribution, such as
the Gaussian distribution $\mu_s$ in \eqref{gauss1}, one then aims to study the evolution in time of certain averaged quantities of the solutions $u(t,x)$. 
One such key quantity is the two-point function $\{N(n,t)\}_{n \in \Z}$, which represents the average energy 
stored at each frequency $n\in \Z$, and are defined by
\begin{align}
N(n,t)=\E[|\ft u_n(t)|^2 ].
\label{twopoint}
\end{align}
\noi
The aim in wave turbulence theory is then to derive an effective equation, called the wave kinetic equation,
for the evolution of the two-point function $\{N(n,t)\}_{n \in \Z}$ and hence the distribution of energy at each frequency; see for example \cite{DH}.

In~\cite[Corollary 1.4]{OTz2}, Oh and Tzvetkov showed that there is a formula for the two-point function \eqref{twopoint}, 
in terms of the Radon-Nikodym derivative of the transported Gaussian measure. This reduces the study of the two-point functions to studying the dynamical properties of the Radon-Nikodym derivative. 
Whilst the result there was stated under the assumption that the Radon-Nikodym derivative belongs to $L^2(d\mu_{s})$, it turns out that it still holds under the weaker assumption $L^1(d\mu_s)$\footnote{We learnt from the authors in \cite{OTz2} that the hypothesis in \cite[Corollary 1.4]{OTz2} may be weakened from $L^2(d\mu_s)$ to $L^1(d\mu_s)$.}.
In the following corollary, we further show that the two-point functions $\{N(n,t)\}_{n \in \Z}$ can be expressed in terms of the explicit formula \eqref{density} for the Radon-Nikodym derivative.
It is interesting then to understand if this explicit formula may be of further use in this setting. 

\begin{corollary}\label{COR:two}
Let $\al\geq 1$, $s>\frac{1}{2}$, and $\Phi_{t}$ denote the flow of FNLS~\eqref{FNLS}. Then, we have	
\begin{align*}
N(n,t)&=\int_{L^2(\T)}  |\ft \phi (n)|^2 \frac{d(\Phi_t)_*\mu_s}{d\mu_s}(\phi) d\mu_s(\phi)\\
&=\int_{L^2(\T)}  |\ft \phi (n)|^2 
\exp\bigg(\mp\int_0^t  \textup{Re} \,\jb{i (|u|^2u)(-t',\phi ), D^{2s}u(-t',\phi)   }_{L^2(\T)}    dt' \bigg) d\mu_s(\phi)
\end{align*}

\noi
for any $n\in \Z$ and $t\in \R$, where $N(n,t)$ is the two-point function defined in \eqref{twopoint}.
\end{corollary}

\subsection{Organization of the paper} 
In Section~\ref{SEC:not}, we introduce some notations and preliminary estimates, such as the linear and bilinear Strichartz estimates. 
In Section~\ref{SEC:proof main}, we present the proof of Proposition \ref{PROP:den} and then prove Theorem~\ref{THM:1} by assuming Lemma~\ref{LEM:energy} and the uniform $L^p$-integrability of the Radon-Nikodym derivative.
The energy estimate (Lemma \ref{LEM:energy}) is proved in Section~\ref{SEC:energy} while we show the uniform $L^p$-integrability in Section \ref{SEC:Lp}. In Section~\ref{SEC:weak}, first by assuming a weaker space-time energy estimate (Lemma~\ref{LEM:energy2}) and using a hybrid argument (combining the approaches in \cite{PTV} and \cite{DT2020}), we present the proof of Theorem \ref{THM:2} and establish the weaker space-time energy estimate (Lemma~\ref{LEM:energy2}).

\section{Notations and preliminary estimates}\label{SEC:not}

\subsection{Notations}
In the following, we fix small $\eps>0$  and set
\begin{align*}
\s = s-\frac{1}{2}-\eps
\end{align*}

\noi
such that \eqref{reg} is satisfied.
Given $N \in  \NN \cup\{\infty\}$, we use $\pi_{\leq N}$ to denote the Dirichlet projection onto the frequencies 
$\{|n| \leq N \}$ and set $\pi_{>N}:=\text{Id}-\pi_{\leq N}$. When $N=\infty$, 
it is understood that $\pi_{\leq N}=\text{Id}$. Given dyadic $M>0$, we let $\P_{M}$ denote the Littlewood-Payley projector onto frequencies $\{|n|\sim M\}$, such that 
\begin{align*}
f=\sum_{\substack{M\geq 1 \\ \text{dyadic}}}^{\infty} \P_{M}f.
\end{align*}
When $M=1$, $\P_{1}$ is a smooth projector to frequencies $\{|n|\les 1\}$.

We  use $a+$ (and $a-$) to denote $a + \eps$ (and $a - \eps$, respectively) for arbitrarily small $\eps \ll 1$,
where an implicit constant is allowed to depend on $\eps > 0$
(and it usually diverges as $\eps \to 0$).

\subsection{Preliminary estimates}
In this subsection, we record some elementary estimates that will be useful in our analysis. 

In our approach, the phase function 
\begin{align}
\Phi_{\al}(\cj n):=\Phi(n_1,n_2,n_3,n_4):=|n_1|^{2\al}-|n_2|^{2\al}+|n_3|^{2\al}-|n_4|^{2\al} \label{p1}
\end{align}
naturally arises as the source of dispersion. 
In order to exploit this for a smoothing benefit, we crucially rely on the following lower bound (see Lemma \ref{lemma: phase lower bound}): for $\al>\tfrac 12$, we have
\begin{align*}
\left|\Phi_{\al}(\overline{n})\right|
& \gtrsim |n_4-n_1||n_4-n_3|n_{\max}^{2\alpha-2} \qquad \text{when} \qquad n_4=n_1-n_2+n_3 
\end{align*}
and where $n_{\max}:=\max( |n_1|,|n_2|,|n_3|,|n_4|)+1$.
This lower bound first appeared in the setting $\frac{1}{2}<\al\leq 1$ in \cite{demirbas2013existence} and the first author and Trenberth established it for $\al>1$ (see Lemma \ref{lemma: phase lower bound}). It can be viewed as a replacement of the explicit factorisations available for the phase function \eqref{p1} of NLS 
\begin{align*}
\Phi_{1}(\cj n)=n_1^{2}-n_2^2+n_3^2-n_4^2=-2(n_4-n_1)(n_4-n_3)  \qquad \text{when }\, n_4=n_1-n_2+n_3,
\end{align*}
and 4NLS (see \cite[Lemma 3.1]{OTz}). 
In particular, we rely on an understanding of the ratio 
\begin{align}
\frac{\Psi_{s}(\cj{n})}{\Phi_{\al}(\cj{n})}, 
\label{ratio1}
\end{align}
 where $\Psi_{s}(\cj{n}):=\jb{n_1}^{2s}-\jb{n_2}^{2s}+\jb{n_3}^{2s}-\jb{n_4}^{2s}.$
A key tool for this analysis is the double mean value theorem; see \cite[Lemma 2.3]{CKSTT}.
\begin{lemma}\label{DMVT}
	Let $\xi,\eta,\lambda\in \R$ and $f\in C^2(\R)$. Then, we have
	\begin{equation*}
	f(\xi+\eta+\lambda)-f(\xi+\eta)-f(\xi+\lambda)+f(\xi)=\lambda\eta\int_0^1\int_0^1f''(\xi+t_1\lambda+t_2\eta)\,dt_1dt_2.
	\end{equation*}
\end{lemma}

\noi
The following result is a direct consequence of Lemma~\ref{DMVT}, using that the map $x\mapsto \jb{x}^{2s-2}$ is non-decreasing when $s\geq 1$.

\begin{lemma}\label{Lemma: DMVT application}
	Fix $s\geq  1$ and let $n_1,n_2,n_3,n_{4}\in \Z$ be such that $n_{4}=n_1-n_2+n_3$. Then, we have
	\begin{equation*}
	|\Psi_{s}(\cj{n})|\lesssim |n_4-n_1| |n_4-n_3|n_{\max}^{2s-2},
	\end{equation*}
	where $n_{\max} = \max(|n_1|,|n_2|,|n_3|,|n_4|)+1$ and the implicit constant depends only on $s$.
\end{lemma}

The use of Lemma~\ref{Lemma: DMVT application} was sufficient to obtain the quasi-invariance results for \eqref{FNLS} in \cite{FT} when $\frac{1}{2}<\al<1$ and $s>1$. 
In this paper, however, we need to refine the analysis in Lemma~\ref{Lemma: DMVT application} to consider $s\leq  1$.

\begin{lemma}\label{LEM:Psis}
	Let $n_1,n_2, n_3,n_4\in \Z$ such that $n_4=n_1-n_2+n_3$. Then:
	
\noi	
\textup{(i)} If $|n_4|\sim |n_1|\sim |n_2|\sim |n_3|$ and $\max(|n_4-n_1|,|n_4-n_3|)\ll |n_3|$, we have 
		\begin{align}
		|\Psi_{s}(\cj{n}) |\les  |n_4-n_1||n_4-n_3| n_{\textup{max}}^{2s-2}, \label{Psis1}
		\end{align}
		for any $s\in \R$.
\medskip
	
\noi
\textup{(ii)} If $|n_4|\sim |n_1|\gg \min(|n_2|,|n_3|)$ and $s\geq \frac{1}{2}$, we have
		\begin{align}
		|\Psi_{s}(\cj{n})|\les  |n_4-n_1|n_{\textup{max}}^{2s-1}. \label{Psis2}
		\end{align}
\end{lemma}
\begin{proof}
The second estimate \eqref{Psis2} follows from the mean value theorem. As for the first estimate \eqref{Psis1}, Lemma~\ref{DMVT} implies 
\begin{align*}
|\Psi_{s}(\cj{n}) &|\les |n_4-n_1||n_4-n_3| \sup_{t_1,t_2\in [0,1]} \big( \jb{n_3 + t_1(n_4-n_1)+t_2(n_4-n_3)}^{2s-2}\big) \\
& \les |n_4-n_1||n_4-n_3|\jb{n_3}^{2s-2},
\end{align*}
	
\noi
thanks to the condition $\max(|n_4-n_1|,|n_4-n_3|)\ll |n_3|$. Hence, we obtain the desired result.
\end{proof}

We now state the following estimate related to the phase function \eqref{p1}.
\begin{lemma}[\cite{demirbas2013existence,FT}]\label{lemma: phase lower bound}
	Fix $\alpha > \tfrac{1}{2}$ and let $n_1,n_2,n_3,n_4\in \Z$ be such that $n_4=n_1-n_2+n_3$. Then, we have
	\begin{align*}
	\left|\Phi_{\al}(\overline{n})\right|&\gtrsim |n_4-n_1||n_4-n_3|\left(  |n_4-n_1|+|n_3-n_1|+|n_4|\right)^{2\alpha-2} \\
	& \gtrsim |n_4-n_1||n_4-n_3|n_{\textup{max}}^{2\alpha-2}
	\end{align*}
	where the implicit constant depends only on $\alpha$.
	In particular, when $\{n_1,n_3\}\neq \{n_4,n_2\}$, the phase function $\Phi_{\al}$ satisfies the following size estimates:
	
\noi
\textup{(i)} If $|n_4|\sim |n_1|\sim |n_2|\sim |n_3|$, then 
		\begin{align*}
		|\Phi_{\al}(\cj{n})| \ges |n_4-n_1||n_4-n_3| n_{\textup{max}}^{2\al-2}. 
		\end{align*}

\noi
\textup{(ii)} If $|n_4|\sim |n_1|\gg \min(|n_2|,|n_3|)$, then 
		\begin{align*}
		|\Phi_{\al}(\cj{n})|\ges |n_4-n_1|n_{\textup{max}}^{2\al-1}. 
		\end{align*}

\noi
\textup{(ii)} If $|n_1|\sim |n_3|\gg \max(|n_4|,|n_2|)$ or $|n_4|\sim |n_2|\gg \max(|n_1|,|n_3|)$, then 
		\begin{align*}
		|\Phi_{\al}(\cj{n})|\ges n_{\textup{max}}^{2\al}. 
		\end{align*}
\end{lemma}

Combining Lemma~\ref{LEM:Psis} and Lemma~\ref{lemma: phase lower bound}, we obtain an estimate on the ratio in \eqref{ratio1}. This result is fundamental in our analysis; see the proof of Lemma \ref{LEM:energy} and Lemma~\ref{LEM:energy2}.

\begin{lemma} \label{LEM:mult}
	Let $\al >\frac{1}{2}$ and $s\geq \frac{1}{2}$. Then, for any $n_1,n_2,n_3,n_4\in \Z$ satisfying $n_4=n_1-n_2+n_3$ with $\{n_1,n_3\}\neq \{n_4,n_2\}$, we have
	\begin{align}
	\frac{ | \Psi_{s}(\cj{n})|}{|\Phi_{\al}(\cj{n})|} \les n_{\textup{max}}^{2s-2\al}. \label{ratio}
	\end{align}
\end{lemma}
\begin{proof}
	Apart from the subcase where $|n_1|\sim |n_2|\sim|n_3|\sim |n_4|$ and $\text{max}(|n_4-n_1|,|n_4-n_3|)\ges |n_3|$, \eqref{ratio} follows from Lemma~\ref{LEM:Psis} and Lemma~\ref{lemma: phase lower bound}. As for this remaining subcase, we note that the mean value theorem implies
	\begin{align*}
	|\Psi_{s}(\cj{n})|\les n_{\text{max}}^{2s-1}\min( |n_4-n_1|,|n_4-n_3|),
	\end{align*}
	and thus 
	\begin{align*}
	\frac{ | \Psi_{s}(\cj{n})|}{|\Phi_{\al}(\cj{n})|}  \les n_{\text{max}}^{2s-2\al+1}\frac{1}{\text{max}(|n_4-n_1|,|n_4-n_3|)}\les n_{\textup{max}}^{2s-2\al}.
	\end{align*}

\noi
This completes the proof of Lemma \ref{LEM:mult}.
\end{proof}

\subsection{Function spaces}\label{SUBSEC:spaces}

We use the Fourier restriction norm spaces $X^{s,b}$, which are adapted to the linear flow of \eqref{FNLS}.
More precisely, given $s,b\in \R$, we define the space $X^{s,b}(\R\times \T)$ via the norm 
\begin{align*}
\|v\|_{X^{s,b}(\R\times \T)}=\| \jb{n}^{s}\jb{\tau -|n|^{2\al} }^{b}\, \ft{v}(\tau, n)\|_{L^{2}_{\tau}\l^{2}_{n}(\R\times \Z)},
\end{align*}
where $\ft{v}(\tau, n)$ denotes the space-time Fourier transform of $v(t,x)$.
 Given $T>0$, we also define the local-in-time version $X^{s,b}([0,T]\times \T)$ of $X^{s,b}(\R\times \T)$ as 
\begin{align*}
X^{s,b}([0,T]\times \T)=\inf\{ \|v\|_{X^{s,b}(\R\times \T)} \, : \, v|_{[0,T]}=u \}.
\end{align*}
We will denote by $X^{s,b}$ and $X^{s,b}_{T}$ the spaces $X^{s,b}(\R\times \T)$ and $X^{s,b}([0,T]\times \T)$, respectively. 
We have the following embedding: for any $s\in \R$ and $b>\frac{1}{2}$, we have
\begin{align*}
X^{s,b}_T \hookrightarrow C([0,T];H^{s}(\T)). 
\end{align*}
Given any function $F$ on $[0, T]\times \T$, we denote by $\tilde{F}$ any extension of $F$ onto $\R\times \T$.

\begin{lemma}\label{LEM:sharp}
Let $s\geq 0$ and $0\leq b<\frac{1}{2}$. Then, for any compact interval $I$, we have 
\begin{align*}
\|\chi_{I}(t)f\|_{X^{s,b}}\les \|f\|_{X^{s,b}},
\end{align*}
where the implicit constant depends only on $b$.
\end{lemma}
\noi
For a proof of Lemma \ref{LEM:sharp}, see for example \cite{DBD}.

\subsection{Linear and bilinear Strichartz estimates}
In this subsection, we present some linear $L_{t,x}^4$-Strichartz estimates on $\T$ as well as a bilinear estimate. We begin with the $L_{t,x}^4$-Strichartz estimate. 

\begin{lemma}\label{LEM:SuperL4}
For $\al\geq 1$ and $b=\frac{1}{4}\big( 1+\frac{1}{2\al}\big)$,
we have 
\begin{align}
\|u\|_{L^{4}_{t,x}(\R\times \T)}& \les \| u\|_{X^{0,b}(\R \times \T)}.\label{L4strich2}
\end{align}
\end{lemma}

We note that the value of $b$ in Lemma~\ref{LEM:SuperL4} is sharp in the sense that the estimate \eqref{L4strich2} fails if $b<\tfrac{1}{4}\big( 1+\tfrac{1}{2\al}\big)$. To see this, one may easily adapt the counterexample~in~\cite[Footnote 9]{OTz}.
The estimate \eqref{L4strich2} is a generalization of the following $L_{t,x}^4$-Strichartz estimate due to Bourgain~\cite{Bou} for  $\al=1$: 
\begin{align*}
\| u\|_{L^{4}_{t,x}(\R\times \T)} \les \|u\|_{X^{0,\frac{3}{8}}(\R\times \T)}. 
\end{align*}
See also \cite{OTz} for a corresponding estimate when $\al=2$. The key difference between the proofs of these estimates and \eqref{L4strich2} is a lack of explicit factorisations due to the fractional powers $|n|^{2\al}$. For this, we need the following counting estimate.

\begin{lemma}
Let $\al\geq 1$, $n\in \Z$, $\tau \in \R$ and $M\geq 1$. Then, we have 
\begin{align}
\# \{ n_1 \in \Z :  | |n_1|^{2\al}+|n-n_1|^{2\al }- \tau |\le M   \} \les_{\al}M^{\frac{1}{2\al}},
\label{Count}
\end{align}
uniformly in $n\in \Z$ and $\tau\in \R$.
\end{lemma}
\begin{proof}
Let $\Psi_{n}(n_1):=|n_1|^{2\al}+|n-n_1|^{2\al}$ and $A(n,\tau ,M)$ be the set in \eqref{Count}.
We claim that it suffices to consider $n>0$. First, if $n=0$, we can argue directly as follows. We may assume $\tau\geq 2 M$, since otherwise we have $|n_1| \les M^{\frac{1}{2\al}}$. Then, $|n_1|\in [ 2^{-\frac{1}{2\al}}(\tau -M)^{\frac{1}{2\al}},  2^{-\frac{1}{2\al}}(\tau +M)^{\frac{1}{2\al}}]$ and hence by the mean value theorem,
\begin{align*}
\# A(0,\tau, M) \leq C(\al)\big[ (\tau+M)^{\frac{1}{2\al}}-(\tau-M)^{\frac{1}{2\al}}\big] \les M^{{\frac{1}{2\al}}}.
\end{align*}
In view of the property $\Psi_{-n}(n_1)=\Psi_{n}(-n_1)$, we may assume $n_1>0$.
We claim that we may also assume $n_1 \geq \frac{n}{2}$. Indeed, if $n_1\in A(n,\tau,M)$, then since $\Psi_{n}(n-n_1)=\Psi_{n}(n_1)$, $n-n_1\in A(n,\tau,M)$ with $n-n_1>\frac{n}{2}$.
Now suppose there are at least two elements in $n_1,n_2\in A(n,\tau, M)$ and without loss of generality, we assume $n_1>n_2$. Since $\al\geq 1$, the function $x\in \R\mapsto \Psi_{n}(x)$ belongs to $C^{2}(\R)$ and we have
\begin{align*}
\Psi_{n}'(n_1)&=2\al |n_1|^{2\al-2}n_1-2\al|n-n_1|^{2\al-2}(n-n_1),\\
\Psi_{n}''(n_1)&=2\al (2\al-1)|n_1|^{2\al -2}+2\al (2\al-1)|n-n_1|^{2\al -2}.
\end{align*}
Since $n_1>\frac{n}{2}$, $\Psi_{n}'(n_1)\geq 0$ and hence by Taylor's formula, we have 
\begin{align*}
\Psi_{n}(n_1)-\Psi_{n}(n_2)& = \Psi_{n}'(n_2)(n_1-n_2)+ \frac 12  \int_0^1 (1-t)\Psi_{n}''(tn_1+(1-t)n_2) (n_1-n_2)^2 dt \\
& \ges_{\al}   \int_0^1 (1-t) \big| n_2+t(n_1-n_2) \big|^{2\al-2} (n_1-n_2)^2 dt \\
&\ges_{\al}|n_1-n_2|^{2\al}.
\end{align*}
Now, for $n_1,n_2\in A(n,\tau, M)$, we have $\Psi_{n}(n_1),\Psi_{n}(n_2)\in [\tau -M, \tau +M]$ and hence
\begin{align*}
|n_2-n_1|\les M^{\frac{1}{2\al}}.
\end{align*}
This completes the proof of \eqref{Count}.
\end{proof}

\begin{proof}[Proof of Lemma~\ref{LEM:SuperL4}]
We follow the standard argument as found in \cite{TAO}. 
For each dyadic number $M$, we define
\begin{align*}
\ft u_M(\tau,n):=\chi_{M\le \jb{\tau -|n|^{2\al}}<2M}  \ft u(\tau, n).
\end{align*}

\noi
Then, by symmetry, we have 
\begin{align*}
\|u\|_{L^4_{t,x}}^2 &=\|uu\|_{L^{2}_{t,x}} \les \sum_{m\ge 0} \sum_{M \ge 0}\| u_{2^mM} u_M \|_{L^2_{t,x}}.
\end{align*}

\noi
Hence, if we have
\begin{align}
\| u_{2^mM} u_M \|_{L^2_{t,x}} &\les (2^m M)^{\frac 1{4\al}} M^{\frac 12}\| u_{2^mM} \|_{L^2_{t,x}} \| u_M \|_{L^2_{t,x}} \notag \\
&= 2^{-\ta m} (2^mM)^{\frac 14(1+\frac{1}{2\al})} M^{\frac 14(1+\frac{1}{2\al})}  \| u_{2^mM}\|_{L^2_{t,x}} \| u_M \|_{L^2_{t,x}}
\label{bili}
\end{align}

\noi
with $\ta:= (\frac 14-\frac 1{8\al})> 0$, then, from the Cauchy-Schwarz inequality, we obtain \eqref{L4strich2}. Therefore, it suffices to show \eqref{bili}. From Plancherel's theorem and the Cauchy-Schwarz inequality, we have
\begin{align}
\| u_{2^mM }u_M \|_{L^2_{t,x}}
&\les  M^{\frac 12} \sup_{(\tau,n)\in \R\times \Z}A(\tau,n, 2^{m}M)^{\frac 12} \| u_{2^mM}  \|_{L^2_{t,x}} \| u_M \|_{L^2_{t,x}}.
\label{bi}
\end{align}  
 It follows from \eqref{Count} that we have
\begin{align}
A(\tau,n, 2^{m}M) \les (2^mM)^\frac 1{2\al}.
\label{mea}
\end{align}

\noi
Thus, from \eqref{bi} and \eqref{mea}, we obtain \eqref{bili}, which completes the proof of Lemma \ref{LEM:SuperL4}.
\end{proof}



Note that by interpolation between 
\begin{align}
X^{0,\frac{1}{2}+} \embeds L^{\infty}_{t}L^{2}_{x}, \label{CTembed}
\end{align}
and $X^{0,0}=L^{2}_{t,x}$,
we have 
\begin{align}
X^{0, \big( \frac{1}{2}-\frac{1}{p}\big)+}\embeds L^{p}_{t}L^{2}_{x}, \label{Xembed1}
\end{align}
for any $2< p <\infty$. Similarly, by interpolation between \eqref{L4strich2} and \eqref{CTembed}, we have 
\begin{align*}
X^{0, \frac {1-4s}4 \big(1+\frac 1{2\al}\big)+2s+ }\embeds L^{4+\frac{16s}{1-4s}}_{t}L^{4-\frac{16s}{1+4s}}_{x},
\end{align*}
for any $0<s<\frac{1}{4}$. Then, the Sobolev embedding $ W^{s,4-\frac{16s}{1+4s}} \embeds L^4$ implies \begin{align}
X^{s, \frac {1-4s}4 \big(1+\frac 1{2\al}\big)+2s+}\embeds L^{4+\frac{16s}{1-4s}}_{t}L^{4}_{x}. \label{Xembed2}
\end{align}
By Bernstein's inequality, we have 
\begin{align}
\|\P_{N}f\|_{L^{\infty}_{t,x}} \les N^{\frac{1}{2}}\|\P_{N}f\|_{X^{0,\frac{1}{2}+}}. \label{Linfty}
\end{align}

\noi

For the weakly dispersive case of \eqref{FNLS} ($\frac 12<\al<1$), the $L^4_{t,x}$-Strichartz estimate looses derivatives due to the weaker curvature of the phase function \eqref{p1}. 

\begin{lemma}{\cite[Corollary 2.11]{ST1}}\label{LEM:StrichWD}
Given $\frac{1}{2}<\al\leq 1$ and $N\gg M$ dyadic, we have
\begin{align}
\|\P_N f\|_{L^{4}_{t,x}}& \les N^{\frac{1-\al}{4}}\|\P_N f\|_{X^{0,\frac{3}{8}}}, \label{L4strich}\\
\|\P_N f \cdot \P_M g\|_{L^2_{t,x}}& \les M^{\frac{1-\al}{2}}\|\P_N f\|_{X^{0,\frac{3}{8}}}\|\P_M g\|_{X^{0,\frac{3}{8}}}, \label{bilinstrich1} 
\end{align}
\end{lemma}

\section{The strongly dispersive case $\al\geq 1$: Proof of Theorem~\ref{THM:1}} \label{SEC:proof main}

In this section, we prove one of our main theorems (Theorem \ref{THM:1}) by assuming Lemma \ref{LEM:energy} and the uniform $L^p$-integrability of the Radon-Nikodym derivative $f_N(t,\cdot)$ \eqref{R1} of the transported measure $(\Phi_{N,t})_*\mu_{s,N,r}$.
We present the proof of Lemma \ref{LEM:energy} and the uniform $L^p$-integrability in Section \ref{SEC:energy} and \ref{SEC:Lp}, respectively. 

\subsection{Proof of Proposition \ref{PROP:den}}\label{SUBSEC:Prop 1.3}

Whilst the proof of Proposition \ref{PROP:den} closely follows~\cite[Section 4]{DT2020},
we include details in order to make this paper self-contained. We recall the basic invariance property of the Lebesgue measures $d\phi_N = \prod_{|n| \le N} d \ft \phi(n)$. 

\begin{lemma}[Liouville's theorem]\label{LEM:Liou}
Let $N\in \NN$. Then, the Lebesgue measure $d\phi_N = \prod_{|n| \le N} d \ft \phi(n)$ is invariant under the flow $\Phi_{N,t}$.
\end{lemma}

\begin{proof}[Proof of Proposition \ref{PROP:den}]

It follows from a change of variables, Lemma \ref{LEM:Liou} and the $L^2$-conservation that for any measurable set $A $ in $E_N$, we have
\begin{align*}
(\Phi_{N,t})_* \mu_{s,N,r}(A)&=Z_{s,N}^{-1}\int_{\phi_N \in \Phi_{N,t}^{-1}(A) } \chi_{ \{  \|\phi_N \|_{L^2} \le r   \}}  e^{-\frac 12 \| \phi_N \|_{H^s}^2  } d\phi_N\\
&=Z_{s,N}^{-1} \int_{\phi_N \in A } \chi_{ \{  \| \Phi_{N,t}^{-1}(\phi_N) \|_{L^2} \le r   \}}   e^{-\frac 12 \| \Phi_{N,t}^{-1}(\phi_N) \|_{H^s}^2 }   d\phi_N\\
&= \int_{\phi_N \in A } \chi_{ \{  \|\phi_N \|_{L^2} \le r   \}}  e^{-\frac 12 \| \Phi_{N,t}^{-1}(\phi_N) \|_{H^s}^2 }   e^{\frac 12 \| \phi_N \|_{H^s}^2 } \mu_{s,N}(d\phi_N)\\
&= \int_{\phi_N \in A }   e^{-\frac 12 \| \Phi_{N,t}^{-1}(\phi_N) \|_{H^s}^2 }   e^{\frac 12 \| \phi_N \|_{H^s}^2 } \mu_{s,N,r}(d\phi_N).
\end{align*}	
	
\noi
Therefore, we have
\begin{align}
f_N(t,\phi_N):=\frac{d(\Phi_{N,t} )_*\mu_{s,N,r} }{d\mu_{s,N,r}}(\phi_N)=e^{-\frac 12 \|  \Phi_{N,t}^{-1}(\phi_N )  \|^2_{H^s}  }e^{\frac 12 \| \phi_N  \|^2_{H^s} }.
\label{Td1}
\end{align}
	
\noi
	Now, by the fundamental theorem of calculus and \eqref{TFNLS},  
\noi
\begin{align*}
f_N(t,\phi_N)&=\exp \bigg(-\frac{1}{2} \int_0^t \frac{d}{dr} \| \Phi_{N,r}^{-1}(\phi_N) \|_{H^s}^2 dr \bigg)\\
&=\exp \bigg(\mp\int_0^t \Re \jb{i(|u_N|^2u_N )(-t',\phi_N) , D^{2s}u_N(-t',\phi_N) }_{L^2(\T)}  dt' \bigg).
\end{align*}
Note that the contribution due to the linear evolution in \eqref{TFNLS} is zero. This proves \eqref{R1}.
	
We now define the natural extension of $f_N(t,\cdot)$ on $L^2(\T)$ by $f_N(t,\phi)=f_N(t,\phi_N)$.
It follows from the uniform $L^p(d\mu_{s,r})$-integrability of the Radon-Nikodym derivative $f_N(t,\cdot)$ that by passing to a subsequence, $f_N(t,\cdot ) $ converges weakly in $L^p(d\mu_{s,r})$. Moreover, from Lemma \ref{LEM:energy},  $f_N(t,\phi)$ converges pointwise to 
\begin{align*}
f(t,\phi)=\exp\bigg(\mp\int_0^t \Re \jb{i (|u|^2u)(-t',\phi), D^{2s}u(-t',\phi)   }_{L^2(\T)}    dt' \bigg)
\end{align*}

\noi
for each $\phi \in L^2(\T)$.
Hence, we obtain 
\begin{align}
f_N(t,\cdot) \wto f(t,\cdot) \qquad \text{in $L^p(d\mu_{s,r})$}
\label{W1}
\end{align}

\noi
as $N\to\infty$ (i.e. $f(t,\cdot)$ is the weak limit of $f_N(t,\cdot)$ in $L^p(d\mu_{s,r})$).

It remains to show that $f(t,\cdot)$ is the Radon-Nikodym derivative of the transported measure $(\Phi_t)_*\mu_{s,r}$ with respect to the Gaussian measure $\mu_{s,r}$ with $L^2$-cutoff. 
It follows from Fubini's theorem and \eqref{Td1} that for any bounded and continuous function $\psi$ on $L^2(\T)$, we have 
\begin{align}
&\int \psi(\phi_N) f_N(t,\phi) \chi_{ \{  \|\phi_N\|_{L^2} \le r  \}  }\mu_s(d\phi) \notag \\
&=\int \psi(\phi_N)f_N(t,\phi)\chi_{ \{  \|\phi_N\|_{L^2} \le r  \}  } \mu_{s,N}(d\phi_N)\otimes \mu_{s,N}^\perp(d\phi_N^\perp) \notag \\
&=\int_{\phi_N^\perp \in E_N^\perp }\bigg\{ \int_{\phi_N \in E_N} \psi(\phi_N)(\Phi_{N,t})_{*}\mu_{s,N,r}(d\phi_N)                 \bigg\} \mu_{s,N}^\perp(d\phi_N^\perp) \notag \\
&=\int_{\phi_N^\perp \in E_N^\perp }\bigg\{ \int_{\phi_N \in E_N} \psi(u_N(t,\phi_N))\mu_{s,N,r}(d\phi_N)                 \bigg\}\mu_{s,N}^\perp(d\phi_N^\perp) \notag\\
&=\int \psi(u_N(t,\phi_N)) \chi_{ \{ \| \phi_N \|_{L^2}\le r  \}   } \mu_s(d\phi). 
\label{X1}
\end{align}
	
\noi
From the approximation property of the truncated dynamics \eqref{LEM:Super5} and the Lebesgue dominated convergence theorem, we have
\begin{align}
\int \psi(u_N(t,\phi_N)) \chi_{ \{ \| \phi_N \|_{L^2}\le r  \}   } \mu_s(d\phi) \too \int \psi(u(t,\phi )) \chi_{ \{ \| \phi \|_{L^2}\le r  \}   } \mu_s(d\phi).
\label{X2}
\end{align}

\noi
By \eqref{W1}, we have
\begin{align}
\int \psi(\phi_N) f_N(t,\phi) \chi_{ \{  \| \phi_N \|_{L^2} \le r  \}  }\mu_s(d\phi) \too \int \psi(\phi) f(t,\phi) \chi_{ \{  \|\phi \|_{L^2} \le r  \}  }\mu_s(d\phi).
\label{X3}
\end{align}
	
\noi
Hence, it follows from \eqref{X1}, \eqref{X2} and \eqref{X3} that we have
\begin{align*}
\int \psi(\phi) f(t,\phi) \chi_{ \{  \|\phi \|_{L^2} \le r  \}  }\mu_s(d\phi)
&=\int \psi(u(t,\phi )) \chi_{ \{ \| \phi \|_{L^2}\le r  \}   } \mu_s(d\phi)=\int \psi(\phi) (\Phi_t)_{*}\mu_{s,r}(d\phi).
\end{align*}
	
\noi
This shows that we obtain
\begin{align*}
\frac{d(\Phi_t)_* \mu_{s,r}}{d\mu_{s,r}}&=f(t,\cdot)
=\exp\bigg(\mp\int_0^t \Re \jb{i (|u|^2u)(-t',\cdot ), D^{2s}u(-t', \cdot)   }_{L^2(\T)}    dt' \bigg).
\end{align*}
	
\noi
This completes the proof of Proposition \ref{PROP:den}.
\end{proof}

\subsection{Proof of Theorem \ref{THM:1}}

We are now ready to present the proof of Theorem \ref{THM:1}.  
Here, we suppose the uniform $L^p(d\mu_{s,r})$-integrability of the Radon-Nikodym derivative $f_N(t,\cdot)$ \eqref{R1} whose proof is presented in Section \ref{SEC:Lp}.

Fix $t\in \R$. Let $A \subset L^2(\T)$ be a measurable set such that $\mu_s(A)=0$. Then, for any $r>0$, we have
\begin{align}
\mu_{s,r}(A)=0.
\label{r1}
\end{align}

\noi
Assume that
$f_N(t,\cdot) \in L^p(d\mu_{s,r}),$
uniformly in $N\in \NN$ (i.e. $f_N(t,\cdot)$ is $L^p(d\mu_{s,r})$-integrable with a uniform (in $N$) bound).
Then, it follows from Proposition \ref{PROP:den} and \eqref{r1} that we obtain
\begin{align}
(\Phi_t)_{*}\mu_{s,r}(A)=\int_A f(t,\phi) \mu_{s,r}(d\phi)=0.
\label{qua1}
\end{align}

\noi
From the Lebesgue dominated convergence theorem and \eqref{qua1}, we have
\begin{align*}
(\Phi_t)_{*}\mu_s(A)
&=\lim_{r\to \infty } \int_{\Phi_{-t}(A)} \chi_{ \{ \|\phi \|_{L^2}\le r  \} } \mu_s(d\phi)\\
&=\lim_{r\to \infty } \int_{\Phi_{-t}(A)} \mu_{s,r}(d\phi)\\
&=\lim_{r\to \infty } (\Phi_t)_{*}\mu_{s,r}(A)=0.
\end{align*} 

\noi
This completes the proof of Theorem \ref{THM:1}.

\section{The strongly dispersive case $\al\geq 1$: the energy estimate}\label{SEC:energy}

In this section, we estimate the Radon-Nikodym derivative \eqref{R1} of the transported measure $(\Phi_{N,t})_*\mu_{s,r,N}$. More precisely, we obtain Lemma \ref{LEM:energy} which will be used to show the uniform $L^p(d\mu_{s,r})$-integrability of the Radon-Nikodym derivative \eqref{R1} in Section \ref{SEC:Lp}.

As a consequence of \cite{Bou}, \cite[Appendix 1]{OTz} and \cite[Appendix 2]{FT}, we have the following global well-posedness and approximation property of solutions to FNLS~\eqref{FNLS} when $\al\geq 1$.
We note that from the $L_{t,x}^4$-Strichartz estimate in Lemma~\ref{LEM:SuperL4}, one can prove properties \eqref{LEM:Super1}, \eqref{LEM:Super3} and \eqref{LEM:Super5}. 

\begin{lemma}[GWP of FNLS~\eqref{FNLS} for $\al\geq 1$~\cite{Bou, OTz, FT}]\label{LEM:Superapprox}
	Let $\al\geq 1$, $\s\geq 0$ and $u_0\in H^{\s}(\T)$. Given $N\in \NB$, there exists a unique global solutions $u_N \in C(\R;H^{\s}(\T))$ to the truncated FNLS~\eqref{TFNLS} with $u_N\vert_{t=0}=\pi_{\leq N}u_0$ and a unique global solution $u\in C(\R;H^{\s}(\T))$ to FNLS~\eqref{FNLS} with $u\vert_{t=0}=u_0$. Moreover, these solutions satisfy:
	\begin{align}
	&\sup_{t\in [-T,T]}\|u(t)\|_{H^{\s}}+\|u\|_{X_{T}^{s,b}}\leq C\|u_0\|_{H^{\s}}, \quad T>0, \label{LEM:Super1}\\
	&\hphantom{XXXXXXXXX}\|u(t)\|_{L^{2}}=\|u_0\|_{L^{2}},\quad t\in \R, \notag\\
	&\sup_{t\in [-T,T]}\|u_N(t)\|_{H^{\s}}+\|u_N\|_{X_{T}^{s,b}}\leq C\|\pi_{\le N}u_0 \|_{H^{\s}}, \quad T>0, \label{LEM:Super3}\\
	&\hphantom{XXXXXXXXX}\|u_N(t)\|_{L^{2}}=\|  \pi_{\leq N}u_0\|_{L^{2}},\quad t\in \R, \notag 
	\end{align}
	where $C>0$ depends only on $T$ and $\|u_0\|_{L^{2}}$ and $b=\frac{1}{2}+\dl$ for fixed $0<\dl\ll 1$ depending only on $\al$. Moreover, 
	for any $T>0$, we have 
	\begin{align}
	\sup_{t\in [-T,T]}\|u(t)-u_{N}(t)\|_{H^{\s}}+\|u-u_N\|_{X^{\s,b}_{T}}\to 0, \label{LEM:Super5}
	\end{align}
	as $N\to \infty$. 
\end{lemma}

We now prove the crucial lemma (Lemma \ref{LEM:energy}) to control the Radon-Nikodym derivative \eqref{R1} of the transported measure $(\Phi_{N,t})_*\mu_{s,N}$.

\begin{lemma}\label{LEM:energy} 
Let $\al\geq 1$, $s\in (\frac{1}{2},1]$ and put $\s =s-\frac{1}{2}-\eps$ for sufficiently small $\eps>0$. Given $N\in \NB $, let $u_{N}\in C(\R;H^{\s}(\T))$ be the solution to \eqref{TFNLS} satisfying $u_{N}\vert_{t=0}=\pi_{\leq N}u_0$, as assured by Lemma~\ref{LEM:Superapprox}. Then, for any $T>0$, we have
\begin{align}
\begin{split}
\bigg| \int_{0}^{T} \Re\, \jb{ i|u_N|^{2}u_N(-t,\cdot),D^{2s}u_N(-t,\cdot)}_{L^2(\T)}dt \bigg|  
&\leq C \|u_N\|_{X_{T}^{0,b}}^{3}\| u_N\|_{X_{T}^{\s,b}}^{3}\\ 
&\le C(\|u_0\|_{L^2},T)\|u_0\|_{H^{\s}}^{3}.
\end{split}
\label{energyest1}
\end{align}

\noi
for some $b=\frac{1}{2}+\dl$, $0<\dl\ll 1$.
Furthermore, if we define $F(u)$ as the functional on the left hand side of \eqref{energyest1} for $u\in C(\R; H^{\s}(\T))$ the solution to FNLS~\eqref{FNLS} with $u\vert_{t=0}=u_0$, then $F(u)$ satisfies the same bound in \eqref{energyest1} and $F(u_N)\to F(u)$ as $N\to \infty$.
\end{lemma}

\begin{proof} 
In the following, we assume \eqref{TFNLS} is focusing; that is, with the positive sign on the nonlinearity.
Note that 
\begin{align}
\begin{split}
\int_{0}^{T} \Re\, \jb{ i|u|^{2}u(-t,\cdot),D^{2s}u(-t,\cdot)}dt& = \text{Im} \int_{0}^{T} \int_{\T} \cj{|u|^{2}u}(-t,x)D^{2s}u(-t,x)dxdt \\
& =\text{Im} \int_{0}^{T} \int_{\T} \big( |u(-t,x)|^{2}-\frac{1}{\pi}\|u(-t,\cdot)\|_{L^{2}}^{2}\big) \\
& \hphantom{XXXXXXXXX}\times\cj{u}(-t,x)D^{2s}u(-t,x)dxdt\\
& = \text{Im} \int_{0}^{T} \int_{\T}\big(  \cj{\N(u)}(-t,x)+\cj{\mathcal{R}(u)}(-t,x)\big) \\
& \hphantom{XXXXXXXXXXXX}\times  D^{2s}u(-t,x)dxdt,
\end{split} \label{energycomp}
\end{align}

\noi 
where 
\begin{align*}
\N(u_1,u_2,u_3)(t,x)&:=\sum_{n\in \Z}\sum_{\substack{n=n_1-n_2+n_3 \\ n_1,n_3\neq n}} \ft u_1 (t,n_1)\cj{ \ft u_2 (t,n_2)} \ft u_3(t,n_3)e^{inx}, \\
\mathcal{R}(u_1,u_2,u_3)(t,x)&:= \sum_{n\in\Z} \ft u_1(t,n)\cj{ \ft u_2(t,n)}\ft u_3(t,n)e^{inx},
\end{align*}
\noi
and $\N(u,u,u)=:\N(u)$ and $\mathcal{R}(u,u,u)=:\mathcal{R}(u)$. 
The second equality in \eqref{energycomp} removes the resonant interactions where $n_1=n$ and $n_2 =n_3 \neq n$ and $n_3=n$ and $n_2 =n_1 \neq n$.
By Parseval's theorem, we see that 
\begin{align*}
\text{Im}\int_{\T} \cj{\mathcal{R}(u)}(-t,x)D^{2s}u(-t,x)dx=0.
\end{align*}
\noi
Therefore, we have
\begin{align*}
\int_{0}^{T} \text{Re} \jb{ i|u|^{2}u(-t,\cdot),D^{2s}u(-t,\cdot)}dt= \text{Im} \int_{0}^{T} \int_{\T}  \cj{\N(u)}(-t,x)D^{2s}u(-t,x)dxdt.
\end{align*}
Let $N\in \NB \cup \{\infty\}$. With $\ft v_N(t,n)=e^{it|n|^{2\al}}\ft u_N(t,n)$ and a symmetrisation argument\footnote{Using $\text{Im} \,a=-\text{Im} \,\cj{a}$.} (as in \cite{HPT, FT}), we have
\begin{align}
&\text{LHS of} \,\, \eqref{energyest1}  \notag \\
& =\textup{Im} \int_{0}^{T} \sum_{n_4\in\Z}\sum_{  \substack{n_4=n_1-n_2+n_3 \\ n_1,n_3\neq n_4}  } e^{-it\Phi_{\al}(\cj{n})}\ft v_N(-t,n_1)\cj{ \ft v_{N}}(-t,n_2)\ft v_N(-t,n_3) \jb{n_4}^{2s}\cj{ \ft v_N}(-t,n_4) dt   \notag \\
& = \frac{1}{4} \textup{Im} \int_{-T}^{0} \sum_{n_4\in\Z}\sum_{  \substack{n_4=n_1-n_2+n_3 \\ n_1,n_3\neq n_4}  }\Psi_{s}(\cj{n}) e^{it\Phi_{\al}(\cj{n})}\ft v_N(t,n_1)\cj{ \ft v_{N}}(t,n_2)\ft v_N(t,n_3) \cj{ \ft v_N}(t,n_4) dt  \notag\\ 
& =:\frac{1}{4} F(v_N).
\label{sym1}
\end{align}

\noi
By symmetry, we may also assume $|n_1|\geq |n_3|$ and $|n_4|\geq |n_2|$.
We perform an integration by parts in time and use \eqref{TFNLS} to obtain
\begin{equation}
\begin{split}
&F(v_N)\\
& = \text{Im}\sum_{n_4\in\Z}\sum_{  \substack{n_4=n_1-n_2+n_3 \\ n_1,n_3\neq n_4}  } \Psi_{s}(\cj{n})\frac{e^{it\Phi_{\al}(\cj{n})}}{i\Phi_{\al}(\cj{n})}\ft v_N(t,n_1)\cj{ \ft v_{N}}(t,n_2)\ft v_N(t,n_3)\cj{ \ft v_N}(t,n_4) dt  \bigg\vert_{t=-T}^{0} \\
&\hphantom{X} - \text{Im}\int_{-T}^{0}\sum_{n_4\in \Z}\sum_{  \substack{n_4=n_1-n_2+n_3 \\ n_1,n_3\neq n_4}  }\Psi_{s}(\cj{n})\frac{e^{it\Phi_{\al}(\cj{n})}}{i\Phi_{\al}(\cj{n})} \dt\big(\ft v_N(t,n_1)\cj{ \ft v_{N}}(t,n_2)\ft v_N(t,n_3) \cj{ \ft v_N}(t,n_4) \big) dt  \\
& = \text{Im}\sum_{n_4\in\Z}\sum_{  \substack{n_4=n_1-n_2+n_3 \\ n_1,n_3\neq n_4}  }\frac{ \Psi_{s}(\cj{n})}{i\Phi_{\al}(\cj{n})}\ft u_N(t,n_1)\cj{ \ft u_{N}}(t,n_2)\ft u_N(t,n_3)\cj{ \ft u_N}(t,n_4) dt  \bigg\vert_{t=-T}^{0}  \\ 
&\hphantom{X}+2\text{Im} \int_{-T}^{0}\sum_{n_4\in \Z}\sum_{  \substack{n_4=n_1-n_2+n_3 \\ n_1,n_3\neq n_4}  }\frac{\Psi_{s}(\cj{n})}{i\Phi_{\al}(\cj{n})}  \cj{ \ft u_{N}}(t,n_2)\ft u_N(t,n_3) \cj{ \ft u_N}(t,n_4)   \\
&\hphantom{XXXXXXXXX}\times\bigg(\sum_{n_1=n_{11}-n_{12}+n_{13}} \ft u_{N}(t,n_{11})\cj{ \ft u_{N}(t,n_{12})} \ft u_{N}(t,n_{13})  \bigg)   dt \\ 
&\hphantom{X}+2\text{Im} \int_{-T}^{0}\sum_{n_4\in \Z}\sum_{  \substack{n_4=n_1-n_2+n_3 \\ n_1,n_3\neq n_4}  }\frac{\Psi_{s}(\cj{n})}{i\Phi_{\al}(\cj{n})}  \ft u_{N}(t,n_1) \ft u_N(t,n_3) \cj{ \ft u_N}(t,n_4) \\
& \hphantom{XXXXXXXXX}\times \cj{\bigg(\sum_{n_2=n_{21}-n_{22}+n_{23}} \ft u_{N}(t,n_{21})\cj{ \ft u_{N}(t,n_{22})} \ft u_{N}(t,n_{23})  \bigg)}   dt \\ 
& =: \N_{0}(u_N)(t)\vert_{t=-T}^{0}+\N_{1}(u_N)+\N_2(u_N).
\label{intebypart}
\end{split}
\end{equation}

\noi
In the following, we will simply write $u$ for $u_N$.
First, we will show 
\begin{align}
|\N_{1}(u)|  &= \bigg\vert 2\text{Re} \int_{-T}^{0}\sum_{n_4\in \Z}\sum_{  \substack{n_4=n_1-n_2+n_3 \\ n_1,n_3\neq n_4}  }\frac{\Psi_{s}(\cj{n})}{\Phi_{\al}(\cj{n})}  \bigg(\sum_{n_1=n_{11}-n_{12}+n_{13}} \ft u(t,n_{11})\cj{ \ft u(t,n_{12})} \ft u(t,n_{13})  \bigg) \notag \\
&\hphantom{XXX}\times \cj{ \ft u}(t,n_2)\ft u(t,n_3) \cj{ \ft u}(t,n_4) dt \bigg\vert \notag \\
& \les \|u\|_{X_{T}^{0,b}}^{3}\| u\|_{X_{T}^{\s,b}}^{3}, \label{Remainderest}
\end{align}
for some $b=\frac{1}{2}+\dl$, $0<\dl\ll 1$.
The same argument can be applied to prove \eqref{Remainderest} for $\N_{2}(u_N)$. Thus, we neglect to show an estimate for $\N_{2}(u_N)$.
In the following, we let $w$ be any extension of $u$ on $[-T,T]$. Namely, $w(t,x)\vert_{[-T,T]}=u(t,x)$.
We split the frequency region into a few cases. We suppose that $|n_{11}|\geq |n_{13}|$ and let $n_{\text{max}}^{(1)}=\max( |n_{11}|,|n_{12}|)+1$.

\medskip
\noi
$\bullet$\underline{ \textbf{Case 1:} $n_{\text{max}}^{(1)} \gg n_{\text{max}}$}

\noi
In this case, $n_{\text{max}}^{(1)} \gg |n_1|$ so $|n_{11}|\sim |n_{12}|$ or $|n_{11}|\sim |n_{13}|$.
Without loss of generality, we assume $|n_{11}|\sim |n_{12}|$.

\medskip
\noi
$\bullet$ \underline{\textbf{Case 1.1:} $|n_{13}|\sim |n_{11}|$}

\noi
We may assume $n_{\text{max}}=|n_3|$.
We write 
\begin{align*}
&\N_1(u) = \sum_{n_4\in \Z} \sum_{\substack{ n_4=n_1-n_2+n_3 \\ n_1=n_{11}-n_{12}+n_{13}}}  \frac{\Psi_{s}(\cj{n})}{\Phi_{\al}(\cj{n})} 
\F_{t}\big\{ ( \chi_{[-T,0]}(t) \cj{ \ft w}(t,n_4)) \cj{\ft w}(t,n_2)  \ft w(t,n_3) \\
&\hphantom{XXXXXXXXXXXXXXXXXXXX}\times \ft w(t,n_{11})\cj{ \ft w}(t,n_{12}) \ft w(t,n_{13})  \big\}(0) \\
& =\intt_{\tau_1-\tau_2+\tau_3-\tau_4 +\tau_5-\tau_6=0}\sum_{n_4\in \Z} \sum_{\substack{ n_4=n_1-n_2+n_3 \\ n_1=n_{11}-n_{12}+n_{13}}}  \frac{\Psi_{s}(\cj{n})}{\Phi_{\al}(\cj{n})}  \ft w(\tau_1,n_{11}) \cj{\ft w}(\tau_2,n_{12}) \ft w(\tau_3,n_{13})  \\
& \hphantom{XXXXXXXXXXXXXXXXXXXX} \times\ft w(\tau_4,n_3) \cj{\ft w}(\tau_5,n_2)  \cj{ \F \{\chi_{[-T,0]} \ft w(\cdot,n_4) \}} (\tau_6).
\end{align*}
Given $\eps>0$, we choose $\dl>0$ such that $\frac{\dl}{1+\dl}=4\eps$.
We define 
\begin{align*}
f&:=\F_{t,x}^{-1}(\jb{n}^{\s}|\ft w(\tau,n)|),\\ 
g&:=\F_{t,x}^{-1}(|\ft w(\tau,n)|), \\
g_{T}&:=\F_{t,x}^{-1}| \F_{t,x} \{ \chi_{[-T,0]} w\}(\tau,n)|.\\ 
\end{align*}

\noi
By Lemma~\ref{LEM:mult}, we have
\begin{align*}
\frac{|\Psi_{s}(\cj{n})|}{|\Phi_{\al}(\cj{n})|} \frac 1{ \jb{n_{11}}^{\s} \jb{n_{12}}^{\s}\jb{n_{13}}^{\s} } & \les \frac{  \jb{n_{\text{max}}}^{2s-2\al }}{ \jb{n_{\text{max}}^{(1)}}^{3\s}} \les \frac{1}{\jb{n^{(1)}_{\text{max}}}^{3\eps} \jb{n_{\text{max}}}^{\nu+6\eps}},
\end{align*}
where $\nu:= 3\s+2\al-2s-9\eps$ and the second inequality follows as $s+2\al \geq \tfrac{3}{2}+12\eps$.
Then, Parseval's theorem implies
\begin{align*}
&|\N_1(u)|\\
&\les \intt_{\tau_1-\tau_2+\tau_3-\tau_4 +\tau_5-\tau_6=0}\sum_{n_4\in \Z} \sum_{\substack{ n_4=n_1-n_2+n_3 \\ n_1=n_{11}-n_{12}+n_{13}}}  \frac{1}{\jb{n^{(1)}_{\text{max}}}^{3\eps} \jb{n_{\text{max}}}^{\nu+6\eps}}  \F_{t,x}f(\tau_1,n_{11}) \F_{t,x}f(\tau_2,n_{12})   \\
& \hphantom{XXXXXXXXXXXXXXXX} \times \F_{t,x}f(\tau_3,n_{13}) \F_{t,x}g(\tau_4,n_3) \F_{t,x}g(\tau_5,n_2) \F_{t,x}g_{T}(\tau_6,n_4) \\
& \les \int_{\R}\int_{\T} |D^{-\frac{\nu}{2}}g_T| |D^{-\frac{\nu}{2}}g| |D^{-6\eps}g| |D^{-\eps}f|^3  dxdt.
\end{align*}

\noi
By using H\"older's inequality, Sobolev's inequality, the $L^{4}_{t,x}$-Strichartz inequality \eqref{Xembed2}, \eqref{Xembed1}, and Lemma \ref{LEM:sharp}, we have
\begin{align*}
\int_{\R}\int_{\T} &|D^{-\frac{\nu}{2}}g_T| |D^{-\frac{\nu}{2}}g| |D^{-6\eps}g| |D^{-\eps}f|^3  dxdt\\
&\les \int_\mathbb{R} \Vert D^{-\frac{\nu}{2}}g \Vert_{L_x^{2+\frac{2}{\dl}}} \Vert D^{-\frac{\nu}{2}} g_T \Vert_{L_x^{2+\frac{2}{\dl}}} \Vert  ( D^{-6\eps}g )(D^{-\eps}f)^{3}  \Vert_{L_{x}^{1+\dl}}dt\\
&\les \int_{\mathbb{R}} \Vert D^{-\frac{\nu}{2} +\frac{1}{2(1+\dl)}  } g \Vert_{L_x^2} \Vert D^{-\frac{\nu}{2} +\frac{1}{2(1+\dl)}  } g_T \Vert_{L_x^2} \Vert D^{-6\eps}g \Vert_{L_x^{\frac{4(1+\dl )}{1-3\dl }}}  \Vert D^{-\eps}f \Vert_{L_x^4}^3dt\\
&\les \Vert g \Vert_{L_t^{\frac{1}{2\eps}} L_x^2 } \Vert g_T \Vert_{L_t^{\frac{1}{2\eps}} L_x^2 }  \Vert D^{-6\eps +\frac{\dl}{1+\dl}} g \Vert_{L_t^{4+\frac{16\eps}{1-4\eps}}L_x^4 } \Vert D^{-\eps}f \Vert^3_{L_t^{4+\frac{16\eps}{1-4\eps}}L_x^4 } \\
&\les \Vert w \Vert^2_{X^{\s,\frac{1}{2}- }} \Vert D^{-5\eps+\frac{\dl}{1+\dl} }g \Vert_{X^{0,\frac{1}{2}-}} \Vert w \Vert_{X^{0,\frac{1}{2}-}}^3\\
&\les \Vert w \Vert_{X^{\s,\frac{1}{2}-}}^3  \Vert w \Vert_{X^{0,\frac{1}{2}-}}^3.
\end{align*}

\noi
Here, we supposed that
\begin{align*}
-\frac{\nu}{2}+\frac{1}{2(1+\dl )}\leq 0 \iff s+2\al\geq \frac{5}{2}+8\eps,
\end{align*}
which is assured by assuming that $\eps>0$ is sufficiently small so that $s\geq \frac{1}{2}+8\eps$.
By taking an infimum over all such extensions $w$ of $u$, we have shown
\begin{align*}
|\N_{1}(u)|\les \|u\|_{X^{0,\frac{1}{2}+}_{T}}^{3} \|u\|_{X^{\s,\frac{1}{2}+}_{T}}^{3}.
\end{align*}

\noi
$\bullet$ \underline{\textbf{Case 1.2:} $|n_{11}|\gg |n_{13}| \gg n_{\text{max}}$}

\noi
In this case, we have 
\begin{align*}
\frac{|\Psi_{s}(\cj{n})|}{|\Phi_{\al}(\cj{n})|} \frac{1}{ \jb{n_{11}}^{\s} \jb{n_{12}}^{\s}\jb{n_{13}}^{\s} }
 &
  \les \frac{1}{\jb{n^{(1)}_{\text{max}}}^{2\eps} \jb{n_{13}}^{\eps}  \jb{n_{\text{max}}}^{\nu+6\eps}},
\end{align*}

\noi
where $\nu=3\s+2\al-2s-9\eps$. Therefore, we have
\begin{align*}
|\N_1(u)| 
& \les \int_{\R}\int_{\T} |D^{-\frac{\nu}{2}}g_T| |D^{-\frac{\nu}{2}}g||D^{-6\eps}g|  |D^{-\eps}f|^3  dxdt
\end{align*}
and hence one can proceed as in Case 1.1 above.

\medskip
\noi
$\bullet$ \underline{\textbf{Case 1.3:} $ |n_{13}| \ll n_{\text{max}}$}

\noi
We may assume $|n_{\text{max}}|\sim |n_4|$. Then, we have 
\begin{align*}
\frac{|\Psi_{s}(\cj{n})|}{|\Phi_{\al}(\cj{n})|} \frac{1}{ \jb{n_{11}}^{\s} \jb{n_{12}}^{\s} \jb{n_4}^{\s} }
  &\les \frac{1}{\jb{n^{(1)}_{\text{max}}}^{2\eps} \jb{n_4}^{\eps} \jb{n_{\text{max}}}^{\nu+6\eps}}
\end{align*}

\noi
where $\nu=3\s+2\al-2s-9\eps$, and hence one can proceed as in Case 1.1. 

 In the following, unless explicitly stated otherwise (for example, as in Subcase 2.2.2), we set $\nu:=3\s+2\al-2s-9\eps$.

\medskip
\noi
$\bullet$ \underline{\textbf{Case 2:} $|n_{\text{max}}^{(1)}|\sim |n_{\text{max}}|$}

\medskip
\noi
$\bullet$ \underline{\textbf{Case 2.1:} $ |n_{\text{max}}| \gg |n_1|$}

\noi
In this case, $|n_\text{max}|\sim |n_4|\sim |n_2|$. 
Then, 
\begin{align*}
\frac{|\Psi_{s}(\cj{n})|}{|\Phi_{\al}(\cj{n})|} \frac{1}{ \jb{n_{11}}^{\s} \jb{n_4}^{\s} \jb{n_2}^{\s} }
&\les \frac{1}{\jb{n_{\text{max}}}^{3\eps}  \jb{n_{\text{max}}}^{\nu+6\eps}}
\end{align*}

\noi
and hence one can proceed as in Case 1.1.

\medskip
\noi
$\bullet$ \underline{\textbf{Case 2.2:} $ |n_{\text{max}}| \sim |n_1|$}

\medskip
\noi
We have $|n_1|\sim |n_3|$ or $|n_1|\sim |n_4|$. We assume $|n_1|\sim |n_3|$, which leads to two subcases. 

\medskip
\noi
$\bullet$ \underline{ \textbf{Subcase 2.2.1:}  $|n_{11}|\gg |n_{12}| $}

\medskip
\noi
In this case, we have $|n_{11}|\sim |n_3|$. Without loss of generality, we assume $|n_4|=\max(|n_{12}|,|n_{13}|,|n_2|, |n_4|)$. If $|n_4|\sim n_{\text{max}}$, then we have
\begin{align*}
\frac{|\Psi_{s}(\cj{n})|}{|\Phi_{\al}(\cj{n})|} \frac{1}{ \jb{n_{11}}^{\s} \jb{n_4}^{\s} \jb{n_3}^{\s} }
  &\les \frac{1}{\jb{n_{\text{max}}}^{3\eps}   \jb{\max(|n_2|,|n_{12}|,|n_{13}|)}^{\nu+6\eps}}
\end{align*}
and we can proceed as in Case 1.1. If $|n_4|\ll n_{\text{max}}$, we only have two frequencies similar to $n_{\text{max}}$, so we need to proceed slightly differently.
We note that
\begin{align*}
\frac{\vert \Psi_s(\overline{n}) \vert }{\vert \Phi_{\al}(\overline{n}) \vert   } \frac{1}{\jb{n_{11}}^{\s }\jb{n_3}^\s  } \lesssim \frac{1}{\jb{n_{\max}}^{2\s+2\alpha-2s}  }= \frac{1}{\jb{n_{\max}}^{\nu}  },
\end{align*} 
where $\nu:=2\s+2\alpha-2s$.
We define
\begin{align*}
f:&=  \F_{t,x}^{-1}(\jb{n}^{\s} | \widehat{w}(\tau,n) | )\\
f_T:&= \F_{t,x}^{-1}( \jb{n}^{\s} | \mathcal{F}\{\chi_{[-T,0]}w \}(\tau,n)| )\\
g:&= \F_{t,x}^{-1}( \jb{n}^{-\eps} | \widehat{w}(\tau,n) |). 
\end{align*}

\noi
Given $\eps>0$, we define $\dl>0$ such that $\frac{\dl}{1+\dl}=4\eps$.
By using H\"older's inequality, Sobolev's inequality, the $L^{4}_{t,x}$-Strichartz inequality \eqref{Xembed2}, \eqref{Xembed1}, and Lemma \ref{LEM:sharp}, we have
\begin{align*}
&\intt_{\tau_1-\tau_2+\tau_3+\dots-\tau_6=0}\sum\limits_{n_4 \in \Z }\sum\limits_{ \substack{n_4=n_1-n_2+n_3\\ n_1=n_{11}-n_{12}+n_{13} }} \Big[ \F_{t,x}g(\tau_1,n_{12}) \F_{t,x}g(\tau_2,n_{13}) \F_{t,x}g(\tau_3,n_{2})\\
&\hphantom{XXXXXXX} \times \Big(\jb{n_{11}}^{-\frac{\nu}{2}} \F_{t,x}f(\tau_4,n_{11}) \Big) \Big( \jb{n_4}^{-\s+3\eps}  \F_{t,x}f(\tau_5,n_4) \Big) \Big(\jb{n_3}^{-\frac{\nu}{2}}  \F_{t,x} f_T(\tau_6,n_3)\Big) \Big] \\
&\lesssim \int_\mathbb{R} \Vert D^{-\frac{\nu}{2}}f \Vert_{L_x^{2+\frac{2}{\dl}}} \Vert D^{-\frac{\nu}{2}} f_T \Vert_{L_x^{2+\frac{2}{\dl}}} \Vert  ( D^{-\s +3\eps}f )g^{3}  \Vert_{L^{1+\dl}}dt\\
&\lesssim \int_{\mathbb{R}} \Vert D^{-\frac{\nu}{2} +\frac{1}{2(1+\dl)}  } f \Vert_{L_x^2} \Vert D^{-\frac{\nu}{2} +\frac{1}{2(1+\dl)}  } f_T \Vert_{L_x^2} \Vert D^{-\s +3\eps} f \Vert_{L_x^{\frac{4(1+\dl )}{1-3\dl }}}  \Vert g \Vert_{L_x^4}^3dt\\
&\lesssim \Vert f \Vert_{L_t^{\frac{1}{2\eps}} L_x^2 } \Vert f_T \Vert_{L_t^{\frac{1}{2\eps}} L_x^2 }  \Vert D^{-\s +3\eps +\frac{\dl}{1+\dl}} f \Vert_{L_t^{4+\frac{16\eps}{1-4\eps}}L_x^4 } \Vert g \Vert^3_{L_t^{4+\frac{16\eps}{1-4\eps}}L_x^4 } \\
&\lesssim \Vert w \Vert^2_{X^{\s,\frac{1}{2}- }} \Vert D^{-\s+4\eps+\frac{\dl}{1+\dl} }f \Vert_{X^{0,\frac{1}{2}-}} \Vert w \Vert_{X^{0,\frac{1}{2}-}}^3\\
&\lesssim \Vert w \Vert_{X^{\s,\frac{1}{2}-}}^3 \Vert w \Vert_{X^{0,\frac{1}{2}-}}^3.
\end{align*}

\noi
Here, we use the following conditions
\begin{align*}
-\s +4\eps+\frac{\dl}{1+\dl} \leq 0\quad \text{and} \quad -\frac{\nu}{2}+\frac{1}{2(1+\dl )}\leq 0.
\end{align*}
The first is justified by choosing $\eps>0$ so small so that $s\geq \frac{1}{2}+9\eps$ and the second follows from $\al\geq 1-\eps$.

\medskip
\noi
$\bullet$ \underline { \textbf{Subcase 2.2.2:} $|n_{11}|\sim  |n_{12}| \sim |n_1| $}

\medskip
\noi
With $|n_4|=\max(|n_{13}|,|n_2|,|n_4|)$, we have
\begin{align*}
\frac{|\Psi_{s}(\cj{n})|}{|\Phi_{\al}(\cj{n})|} \frac{1}{ \jb{n_{11}}^{\s} \jb{n_{12}}^{\s} \jb{n_3}^{\s} }
  &\les \frac{1}{\jb{n_{\text{max}}}^{3\eps} \jb{n_4}^{\nu+6\eps}}
\end{align*}

\noi
and hence we proceed as in Case 1.1.

\medskip
\noi
$\bullet$ \underline{\textbf{Case 3:} $|n_{\text{max}}^{(1)}|\ll |n_{\text{max}}|$}

\medskip
\noi
Since $|n_1| \les n_{\text{max}}^{(1)}$, we have $n_{\text{max}}\sim |n_4| \sim |n_2|$.
Suppose $|n_3|\sim n_{\text{max}}$. Then, we have 
\begin{align*}
\frac{|\Psi_{s}(\cj{n})|}{|\Phi_{\al}(\cj{n})|} \frac{1}{ \jb{n_4}^{\s} \jb{n_{2}}^{\s} \jb{n_3}^{\s}}
  &\les \frac{1}{\jb{n_{\text{max}}}^{3\eps}  \jb{n_{11}}^{\nu+6\eps} },
\end{align*}

\noi
and hence we proceed as in Case 1.1. Otherwise, we have $|n_3|\ll n_{\text{max}}$.
 In this case, without loss of generality, we assume $|n_{3}|=\max( |n_{11}|,|n_{12}|,|n_{13}|,|n_3|).$ 
 We note that
\begin{align*}
\frac{\vert \Psi_s(\overline{n}) \vert }{\vert \Phi_{\al}(\overline{n}) \vert   } \frac{1}{\jb{n_4}^{\s }\jb{n_2}^\s  } \lesssim \frac{1}{\jb{n_{\max}}^{2\s+2\alpha-2s}  }= \frac{1}{\jb{n_{\max}}^{\nu}  },
\end{align*} 
where $\nu:=2\s+2\alpha-2s$. Hence, we can proceed as in Subcase 2.2.1.

To finish the proof of \eqref{energyest1}, it remains to estimate the boundary term $\N_{0}$ in \eqref{intebypart}. Fix $t\in \R$. We denote by $n_{(j)}$ the $j$-th largest frequency among $(n_1,n_2,n_3,n_4)$.
\medskip

\noi
$\bullet$ \underline{ \textbf{Case 1:} $|n_{(1)}|\sim |n_{(2)}|\sim |n_{(3)}|\ges |n_{(4)}|$}

\medskip
\noi

Without loss of generality, we may assume $|n_1|\sim|n_3|\sim|n_2|$.
We define $f(n)=\jb{n}^{\s}|\ft u(t,n)|$ and using Lemma~\ref{LEM:mult}, we have 
\begin{align*}
|\N_0(u)(t)|& \les \sum_{n_4\in\Z}\sum_{  \substack{n_4=n_1-n_2+n_3 \\ n_1,n_3\neq n_4}  } \frac{1}{\jb{n_{\text{max}}}^{\nu}} f(n_1)f(n_2)f(n_3)|\ft u(t, n_4)| \\
& \les \bigg(\sum_{n_4\in\Z}\sum_{  n_4=n_1-n_2+n_3 }  \frac{ f(n_1)^{2}}{\jb{n_3}^{\nu}} |\ft u(t,n_4)|^2 \bigg)^{\frac{1}{2}}\bigg(\sum_{n_4\in\Z}\sum_{  n_4=n_1-n_2+n_3 }  \frac{f(n_2)^{2} }{\jb{n_1}^{\nu}} f(n_3)^{2}\bigg)^{\frac{1}{2}} \\
&\les \|u(t)\|_{L^2}\|u(t)\|_{H^{\s}}^{3},
\end{align*}
where $\nu:=3\s+2\al-2s$ and the sums above converge, since if $s>\frac{5}{2}-2\al+3\eps$, we ensure that 
$\nu>1$, for $\s=s-\frac{1}{2}-\eps$.

\medskip
\noi
$\bullet$ \underline {\textbf{Case 2:} $|n_{(1)}|\sim |n_{(2)}|\gg |n_{(3)}|\ges  |n_{(4)}|$}

\medskip
\noi

Without loss of generality, we may assume $|n_1|\sim |n_4|\gg |n_3| \geq |n_2|$. 
From Lemma~\ref{LEM:mult}, we have 
\begin{align*}
|\N_0(u)(t)|& \les \sum_{n_4\in\Z}\sum_{  \substack{n_4=n_1-n_2+n_3 \\ n_1,n_3\neq n_4}  } \frac{1}{\jb{n_{\text{max}}}^{\nu}}\frac{1}{\jb{n_3}^{\s}} f(n_1)|\ft u (t,n_2)|f(n_3) f(n_4),
\end{align*}
where $\nu:=2\s+2\al-2s$ which is positive as $\al>\frac{1}{2}$. By Cauchy-Schwarz inequality, we bound the above by
\begin{align*}
& \bigg(\sum_{n_4\in\Z}\sum_{  n_4=n_1-n_2+n_3 } f(n_3)^{2}  f(n_1)^2  f(n_4)^2 \bigg)^{\frac{1}{2}}\bigg(\sum_{n_4\in\Z}\sum_{  n_4=n_1-n_2+n_3 }  \frac{|\ft u(t,n_2)|^2 }{\jb{n_{\text{max}}}^{2\nu}\jb{n_3}^{2\s}}   \bigg)^{\frac{1}{2}} \\
& \les \|u(t)\|_{L^2}\|u(t)\|_{H^\s}^{3} \bigg( \sum_{n_1, n_3} \frac{1}{\jb{n_1}^{1+\eps}\jb{n_3}^{1+\eps}}\frac{1}{\jb{n_1}^{2\nu+2\s-2-2\eps}}\bigg)^{\frac{1}{2}}
\end{align*}
and the above sums converge since $s>\frac{5}{2}-2\al+3\eps$ implies $\nu+\s-1-\eps\geq 0$. We also used that $s \le 1+ \frac {3\eps}{2}$ so the condition $-2\s+1+\eps \ge 0$ is satisfied.

Now, we prove $F(u_N)\to F(u)$ as $N \to \infty$. 
We set $w_N = v - v_N$. From \eqref{sym1}, we have
\begin{align*}
&\big|F(u)-F(u_N)\big|\\
&= \bigg|  \frac{1}{4} \Im \int_{-T}^{0} \sum_{n_4\in\Z}\sum_{  \substack{n_4=n_1-n_2+n_3 \\ n_1,n_3\neq n_4}  }\Psi_{s}(\cj{n}) e^{it\Phi_{\al}(\cj{n})} \bigg(\ft w_N(t,n_1)\cj{ \ft v}(t,n_2)\ft v(t,n_3) \cj{ \ft v}(t,n_4)\\
&\hphantom{XXX} +\ft v_N(t,n_1)\cj{ \ft w_N}(t,n_2)\ft v(t,n_3) \cj{ \ft v}(t,n_4) +\ft v_N(t,n_1)\cj{ \ft v_N}(t,n_2)\ft w_N(t,n_3) \cj{ \ft v}(t,n_4)\\
&\hphantom{XXX}+ \ft v_N(t,n_1)\cj{ \ft v_N}(t,n_2)\ft v_N(t,n_3) \cj{ \ft w_N}(t,n_4)       \bigg) dt      \bigg|.
\end{align*}

\noi
Note that
\begin{align*}
&e^{it|n_1|^{2\al} }\dt \ft {w}_N(t,n_1)\\
&=\sum_{n_1=n_{11}-n_{12}+n_{13}} \bigg( \Big(\ft u(t,n_{11})-\ft u_N(t,n_{11}) \Big)   \cj{ \ft u}(t,n_{12}) \ft u(t,n_{13})\\
& +  \ft u_N(t,n_{11})  \cj{ \Big(\ft u(t,n_{12})-\ft u_N(t,n_{12}) \Big) }   \ft u(t,n_{13}) +  \ft u_N(t,n_{11}) \cj{ \ft u_N(t,n_{12}) }   \Big(\ft u(t,n_{13})-\ft u_N(t,n_{13}) \Big)      \bigg).
\end{align*}

\noi
Therefore, by proceeding as in the proof of \eqref{energyest1} and from Lemma \ref{LEM:Superapprox}, we have
\begin{align*}
\big| F(u)-F(u_N)  \big|&\les 
\sup_{t\in [0,T]}\big(  \|u(t)\|_{H^{\s}}+\| u_N  (t)\|_{H^{\s}} \big)^3 \| u(t)-u_N(t)\|_{H^{\s}}\\
&\hphantom{XX} +\big(  \|u\|_{X_T^{\s,\frac 12+}}+\| u_N \|_{X_T^{\s,\frac 12+}} \big)^5 \| u-u_N\|_{X_T^{\s,\frac 12+}}  \too 0
\end{align*}

\noi
as $N\to \infty$. This completes the proof of Lemma \ref{LEM:energy}.  
\end{proof}

\section{Uniform $L^p$-integrability of the Radon-Nikodym derivative}\label{SEC:Lp}

In this section, we prove the uniform $L^p$-integrability of the Radon-Nikodym derivative $f_N(t,\cdot)$ in \eqref{R1}.  
It follows from \eqref{R1} and Lemma \ref{LEM:energy} that we have
\begin{align}
\int   |f_N(t,\phi)|^p \mu_{s,r}(d\phi)
&= \int  \chi_{ \{ \| \phi \|_{L^2} \le r \}}  |f_N(t,\phi)|^p \mu_s(d\phi) \notag\\
&\le \int \chi_{  \{ \| \phi \|_{L^2} \le r  \}   }\exp(C  \| \phi_N  \|_{H^\s}^{3})   \mu_s(d\phi), \label{expbound}
\end{align}

\noi
where $C=C(t,p,r)$ depends only on $t,p$, and $r$. 
Therefore, to show the uniform $L^p$-integrability, it suffices to prove the following lemma.

\begin{lemma}\label{LEM:un1} Let $\frac{1}{2}<s\leq 1$, $\s=s-\frac{1}{2}-\eps$ for small $\eps>0$ and $q\geq 1$ such that $\s q<s$. Then, for any $r > 0$,  we have the following exponential integrability 
\begin{align}
\sup_{N \in \NN}  \E_{\mu_s} \Big[ \chi_{\{\int_{\T} |\phi_N|^2 dx \, \le r\}}   e^{ ( \int_\T  |D^\s \phi_N |^2 dx)^q  } \Big] < \infty. 
\label{Gibbs4}
\end{align}
	
\end{lemma}

Note that given $\frac{1}{2}<s\leq 1$, \eqref{Gibbs4} holds for $q=\frac{3}{2}$. This implies the uniform boundedness of \eqref{expbound} and hence completes the proof of Proposition~\ref{PROP:den}.
A similar, slightly more general statement (replacing the $H^{\s}$-norm in \eqref{Gibbs4} with a Besov norm) was proved in \cite[Lemma 2.2 and Corollary 2.3]{DT2020}, adapting arguments in~\cite{Bo94}.
Our proof of Lemma \ref{LEM:un1} is based on the variational approach
due to Barashkov and Gubinelli~\cite{BG}. 
More precisely, we will rely on 
the Bou\'e-Dupuis variational formula \cite{BD, Ust}; see Lemma~\ref{LEM:var3}.

\subsection{Variational formulation}
\label{SUBSEC:var}

In order to  prove \eqref{Gibbs4}, 
we use a variational formula for the partition function
as in \cite{TW, OOT, OST20}.
Let us first introduce some notations.
Let $W(t)$ be a cylindrical Brownian motion in $L^2(\T)$.
Namely, we have
\begin{align*}
W(t) = \sum_{n \in \Z} B_n(t) e^{inx},
\end{align*}

\noi
where  
$\{B_n\}_{n \in \Z}$ is a sequence of mutually independent complex-valued\footnote{By convention, we normalize $B_n$ such that $\text{Var}(B_n(t)) = t$. In particular, $B_0$ is  a standard real-valued Brownian motion.} Brownian motions.
Then, define a centered Gaussian process $Y(t)$
by 
\begin{align}
Y(t)
=  \jb{\nabla}^{-s}W(t).
\label{P2}
\end{align}

\noi
Note that 
we have $\Law_\PP(Y(1)) = \mu_s$, 
where $\mu_s$ is the Gaussian measure in \eqref{gauss1}.
By setting  $Y_N = \pi_NY $, 
we have   $\Law_\PP(Y_N(1)) = (\pi_N)_*\mu_s$, 
i.e.~the push-forward of $\mu_s$ under $\pi_N$.
In particular, a typical function $u$ in the support of $\mu_s$ belongs to $L^\infty(\T)$ when $s>\frac 12$.

Next, let $\Ha$ denote the space of drifts, 
which are progressively measurable processes 
belonging to 
$L^2([0,1]; L^2(\T))$, $\PP$-almost surely. 
We now state the  Bou\'e-Dupuis variational formula \cite{BD, Ust};
in particular, see Theorem 7 in \cite{Ust}.

\begin{lemma}\label{LEM:var3}
	Let $Y$ be as in \eqref{P2}.
	Fix $N \in \NN$.
	Suppose that  $F:C^\infty(\T^d) \to \R$
	is measurable such that $\E\big[|F(\pi_NY(1))|^p\big] < \infty$
	and $\E\big[|e^{-F(\pi_NY(1))}|^q \big] < \infty$ for some $1 < p, q < \infty$ with $\frac 1p + \frac 1q = 1$.
	Then, we have
	\begin{align*}
	- \log \E\Big[e^{-F(\pi_N Y(1))}\Big]
	= \inf_{\dr \in \mathbb H_a}
	\E\bigg[ F(\pi_N Y(1) + \pi_N I(\dr)(1)) + \frac{1}{2} \int_0^1 \| \dr(t) \|_{L^2_x}^2 dt \bigg], 
	\end{align*}

\noi
where  $I(\dr)$ is  defined by 
\begin{align*}
I(\dr)(t) = \int_0^t \jb{\nabla}^{-s} \dr(t') dt'
\end{align*}

\noi
and the expectation $\E = \E_\PP$
is an 
expectation with respect to the underlying probability measure~$\PP$.

\end{lemma}

Before proceeding to the proof of Lemma~\ref{LEM:un1}, 
we state a lemma on the pathwise regularity bounds  of 
$Y(1)$ and $I(\dr)(1)$.

\begin{lemma}  \label{LEM:Dr}
	
\textup{(i)} 
For any finite $p \ge 1$, we have 
\begin{align}
\E \Big[  \| Y_N(1)\|_{W^{\s,\infty} }^p \Big] \leq C_p <\infty,
\label{u1}
\end{align}

\noi
uniformly in $N \in \NN$.
	
\smallskip
	
\noi
\textup{(ii)} For any $\dr \in \Ha$, we have
\begin{align}
\| I(\dr)(1) \|_{H^{s}}^2 \leq \int_0^1 \| \dr(t) \|_{L^2}^2dt.
\label{CM}
\end{align}
\end{lemma}

\begin{proof}
It follows from the Sobolev embedding theorem that for any sufficiently small $0<\eta\ll 1$, there exists finite $r=r(\eta)>0$ such that
\begin{align*}
\| u \|_{W^{s,\infty}(\T) }\les \| u \|_{W^{s+\eta,r}( \T )}.
\end{align*}

\noi
Here, we reduce the $r=\infty$ case to the case of large but finite $r$ by paying the expense of a slight loss of spatial derivative. 
Hence, we assume $r<\infty$ in the following.

Let $p \geq r$.
Note that  
\begin{align*}
\jb{\nabla}^{\s+\frac \eps 2}Y_N(1)=\sum_{|n|\le N} \frac{B_n(1)}{\jb{n}^{\frac 12 +\frac \eps 2 }}e^{inx}.
\end{align*}

\noi
Then, one can see that $\jb{\nabla}^{\s+\frac \eps 2}Y_N(1)$ is a mean zero Gaussian random variable with variance $\| \jb{n}^{-\frac{1}{2}-\frac{\eps}{2}}\chi_{|n|\leq N}\|_{\l^2_{n}}^{2}$.
Hence, there exists a constant $C>0$ such that 
\begin{align}
\|  \jb{\nb}^{\s+\frac \eps 2}Y_N(1)  \|_{L^p_\o }\le Cp^{\frac 12} \| \jb{n}^{-\frac{1}{2}-\frac{\eps}{2}}\chi_{|n|\leq N}\|_{\l^2_{n}}. 
\label{w1}
\end{align} 

\noi
Then, from Minkowski's inequality and \eqref{w1}, we have
\begin{align*}
\big(\E \| Y_N(1) \|_{W_x^{\s+\frac \eps 2,r} }^p \big)^{\frac 1p} &\le  \big\|  \|   \jb{\nb}^{\s+\frac \eps 2}Y_N(1)     \|_{L^p_\o} \big\|_{L_x^r} \le Cp^{\frac 12}\big( \sum_{ | n | \le  N}\frac 1{\jb{n}^{1+\eps}}  \big)^\frac 12 \le \wt{C}p^{\frac 12}. 
\end{align*}

\noi
This proves \eqref{u1}.

As for (ii), 
the estimate \eqref{CM}
follows  from Minkowski's and Cauchy-Schwarz' inequalities. 
See the proof of Lemma 4.7 in \cite{GOTW}.

\end{proof}

\subsection{Uniform exponential integrability}

In this section, we present the proof of Lemma \ref{LEM:un1}. 
Since the argument is identical for any finite $p \geq 1$, we only present details for the case $p =1$. 
Note that
\begin{align}
\chi_{\{|\,\cdot \,| \le K\}}(x) \le \exp\big( -  A |x|^\gamma\big) \exp(A K^\g)
\label{H6}
\end{align}

\noi
for any $K, A , \g > 0$.
Given $N \in \NN$ and $A\gg 1$ sufficiently large, we define 
\begin{align}
\begin{split}
R_N (u)
&=  \bigg(  \int_{\T}   |D^\s u_N |^2  dx \bigg)^q
- A \, \bigg( \int_{\T} | u_N|^2  dx\bigg)^\g 
\end{split}
\label{K2}
\end{align}

\noi
for some $\g> q$ which will be determined later (see \eqref{ga}). 
Then, the following uniform exponential bound \eqref{u2} with \eqref{H6}
\begin{equation}
\sup_{N\in \NN} \Big\| 
e^{R_N(u)}\Big\|_{L^p( \mu_s)}
\leq C_{p,A,\g}  < \infty
\label{u2}
\end{equation}

\noi
implies the uniform exponential integrability \eqref{Gibbs4}. Hence, it remains to prove the uniform exponential integrability \eqref{u2}.
In view of the Bou\'e-Dupuis formula (Lemma \ref{LEM:var3}), 
it suffices to  establish a  lower bound on 
\begin{equation}
\W_N(\dr) = \E
\bigg[-R_N(Y(1) + I(\dr)(1)) + \frac{1}{2} \int_0^1 \| \dr(t) \|_{L^2_x}^2 dt \bigg], 
\label{v_N0}
\end{equation}

\noi 
uniformly in $N \in \NN$ and  $\dr \in \Ha$.
We  set $Y_N = \pi_N Y = \pi_N Y(1)$ and $\Dr_N = \pi_N  \Dr = \pi_N I(\dr)(1)$.

From \eqref{K2}, we have
\begin{align}
\begin{split}
R_N (Y + \Dr)  & = 
\bigg(\int_{\T}   |D^\s Y_N|^2  dx
+ 2\int_{\T}  \Re ( \cj{D^\s Y_N}  D^\s \Dr_N )  dx+ \int_{\T}     |D^\s \Dr_N|^2 dx \bigg)^q
\\
&\hphantom{X}
-  A \bigg\{ \int_{\T} \Big(  |Y_N|^2  + 2 \Re(\cj{Y_N} \Dr_N) + |\Dr_N|^2 \Big) dx \bigg\}^\g. 
\end{split}
\label{Y0}
\end{align}

\noi
Hence, from  \eqref{v_N0} and \eqref{Y0}, we have
\begin{align}
\begin{split}
\W_N(\dr)
&=\E
\bigg[
-\bigg(\int_{\T}   |D^\s Y_N|^2  dx
+ 2\int_{\T}  \Re (\cj{D^\s Y_N}  D^\s \Dr_N )  dx+ \int_{\T}     |D^\s \Dr_N|^2 dx \bigg)^q
\\
&\hphantom{XXX}
+A \bigg\{ \int_{\T} \Big(  |Y_N|^2  + 2 \Re( \cj{Y_N} \Dr_N) + |\Dr_N|^2 \Big) dx \bigg\}^\g
+ \frac{1}{2} \int_0^1 \| \dr(t) \|_{L^2_x}^2 dt 
\bigg].
\end{split}
\label{v_N0a}
\end{align}

We first state 
a lemma, controlling the terms appearing in \eqref{v_N0a}.
We present the proof of this lemma at the end of this section.

\begin{lemma} \label{LEM:Dr2}
	
	Let $s$, $\s$ and $q$ be as given in Lemma~\ref{LEM:un1}. Then, we have the following:
	
	\noi
\textup{(i)}
There exists $c >0$ such that
\begin{align}
\|  \Dr_N \|_{H^\s }^{2q}
&\le \frac 1{100}
\| \Dr_N \|_{H^s}^2 + \frac A{100} \| \Dr_N \|_{L^2}^{ \frac{2q(1+2\eps)}{q(1+2\eps) -2s(q-1)}},
\label{YY3}\\
\bigg(\int_{\T} \big| \cj{ D^\s Y_N} D^\s \Dr_N \big| dx\bigg)^q 
&\le c \|  Y_N  \|_{W^{\s,\infty}  }^{2q} 
+ \frac 1{200} 
\| \Dr_N \|_{H^s}^2+\frac {A}{200} 
\| \Dr_N \|_{L^2}^{ \frac{2q(1+2\eps)}{q(1+2\eps) -2s(q-1)}}
\label{YY2} 
\end{align}

\noi  
for any sufficiently large $A>0$, uniformly in $N \in \NN$.

\smallskip

\noi
\textup{(ii)}	
Let $A, \g > 0$. Then,  
there exists $c = c(A,\g)>0$ such that
\begin{align}
\begin{split}
A \bigg\{ \int_{\T}& \Big(  |Y_N|^2  + 2 \Re( \cj{Y_N} \Dr_N) + |\Dr_N|^2 \Big) dx \bigg\}^\g \ge \frac A4 \| \Dr_N \|_{L^2}^{2\g} 
- c \| Y_N \|_{L^\infty }^{2\g}, 
\end{split}
\label{YY5}
\end{align}

\noi
uniformly in $N \in \NN$.

\end{lemma}

Set
\begin{align}
\g=\frac {q(1+2\eps)}{q(1+2\eps)-2s(q-1)}.
\label{ga}
\end{align}
Note that to ensure the denominator is non-vanishing, we impose $\s q<s$.
Then, as in \cite{BG, GOTW, ORSW2, OOT}, 
the main strategy is 
to establish a pathwise lower bound on $\W_N(\dr)$ in~\eqref{v_N0a}, 
uniformly in $N \in \NN$ and $\dr \in \Ha$, 
by making use of the 
positive terms:
\begin{equation}
\U_N(\dr) =
\E \bigg[\frac A 4\| \Dr_N\|_{L^2}^{2\g} + \frac{1}{2} \int_0^1 \| \dr(t) \|_{L^2_x}^2 dt\bigg].
\label{v_N1}
\end{equation}

\noi
coming from \eqref{v_N0a} and \eqref{YY5}.
From \eqref{v_N0a} and \eqref{v_N1} together with Lemmas  \ref{LEM:Dr2} and \ref{LEM:Dr}, we obtain
\begin{align}
\inf_{N \in \mathbb{N}} \inf_{\dr \in \Ha} \W_N(\dr) 
\geq 
\inf_{N \in \mathbb{N}} \inf_{\dr \in \Ha}
\Big\{ -C_0 + \frac{1}{10}\U_N(\dr)\Big\}
\geq - C_0 >-\infty.
\label{YYY5}
\end{align}

\noi
Then,  
the uniform exponential integrability \eqref{u2} 
follows from 
\eqref{YYY5} and Lemma \ref{LEM:var3}.
This completes the proof of  
Lemma \ref{LEM:un1}.

\medskip

We conclude this section by 
presenting the proof of Lemma \ref{LEM:Dr2}.

\begin{proof}[Proof of Lemma \ref{LEM:Dr2}]
	(i)
	It follows from interpolation and Young's inequality that 
\begin{align*}
\begin{split}
\|  \Dr_N \|_{H^\s }^{2q}
&\les  \| \Dr_N \|_{H^s}^{\frac {2q\s}s } \| \Dr_N \|_{L^2}^{2q(1-\frac \s s)}   \le \frac 1{100} \| \Dr_N \|_{H^s}^2 + \frac A{100} \| \Dr_N \|_{L^2}^{2\g},
\end{split}
\end{align*}
where $A>0$ is sufficiently large. 
Here, the second inequality follows
from Young's inequality since
$\frac{q\s}{s} + \frac {q}{\g}(1-\frac \s s )=1$, where $\g$ is given in \eqref{ga}. 
This yields the first estimate \eqref{YY3}.
As for the second estimate \eqref{YY2}, the Cauchy-Schwarz and Cauchy inequalities imply
\begin{align*}
\bigg(\int_{\T} \big| \cj{D^\s Y_N} D^\s \Dr_N \big| dx\bigg)^q
&  \le  \|  Y_N\|_{H^{\s}}^q
\|\Dr_N\|_{H^\s}^q \le \frac{1}{2} \|  Y_N\|_{W^{\s,\infty}}^{2q}+
\frac{1}{2}\|\Dr_N\|_{H^{\s}}^{2q}.
\end{align*}
Now we apply \eqref{YY3}.

\medskip
	
\noi
(ii) 
Note that there exists a constant $C_\g>0$ such that 
\begin{align}
|a+b+c|^\g
\ge \frac 12 |c|^\g - C_\g(|a|^\g+|b|^\g)
\label{OOT}
\end{align}
	
\noi
for any $a,b,c \in \R$ (see Lemma 5.8 in \cite{OOT}). Then, from \eqref{OOT}, we have 
\begin{align}
\begin{split}
A \bigg\{ \int_{\T}& \Big(  |Y_N|^2  + 2 \Re( \cj{Y_N} \Dr_N) + |\Dr_N|^2 \Big) dx \bigg\}^\g \\
&\ge \frac A2 \bigg( \int_{\T} |\Dr_N|^2dx \bigg)^\g
- AC_\g \bigg\{ \bigg( \int_{\T} | Y_N|^2  dx \bigg)^\g
+ \bigg( \int_{\T} |\cj{Y_N} \Dr_N| dx \bigg)^\g \bigg\}.
\end{split}
\label{YZ3}
\end{align}

\noi
From Young's inequality, we have
\begin{align}
\begin{split}
\bigg( \int_{\T} | \cj{Y_N} \Dr_N| dx \bigg)^\g
&\le \| Y_N \|_{L^\infty}^\g \| \Dr_N \|_{L^2}^\g  \le c \| Y_N \|_{L^\infty}^{2\g}  + \frac1{100C_\g} \| \Dr_N \|_{L^2}^{2\g}.
\end{split}
\label{YZ4}
\end{align}
	
\noi
Hence, \eqref{YY5} follows from \eqref{YZ3} and \eqref{YZ4}.
\end{proof}

\section{The weakly dispersive case $\frac{1}{2}<\al <1$}\label{SEC:weak}

In this section, we consider the weakly dispersive FNLS~\eqref{FNLS} with $\frac{1}{2}<\al<1$ and prove our second main result (Theorem~\ref{THM:2})

\subsection{Truncated dynamics in the weakly dispersive case}

We consider two kinds of approximating flows:
\begin{equation}\label{TFNLS21}
\begin{cases}
i\partial_t u_N+(-\dx^2)^\alpha u_N=\pm \pi_{\leq N}\big( |u_N|^2u_N\big), \quad (t,x)\in \R\times \T,\\
u|_{t=0}=\pi_{\leq N}\phi, 
\end{cases}
\end{equation} 
and
\begin{equation}\label{TFNLS22}
\begin{cases}
i\partial_t u_N+(-\dx^2)^\alpha u_N=\pm \pi_{\leq N}\big( |\pi_{\leq N} u_N|^2  \pi_{\leq N}u_N\big), \quad (t,x)\in \R\times \T,\\
u|_{t=0}=\phi. 
\end{cases}
\end{equation} 
The difference between these is that \eqref{TFNLS22} is the linear evolution on 
high frequencies $\{ |n|>N\}$, whilst \eqref{TFNLS21} vanishes on high frequencies. 
We denote the solution maps of \eqref{TFNLS21} and \eqref{TFNLS22} at time $t>0$ by $\wt{\Phi}_{N,t}$ and $\Phi_{N,t}$ respectively.
We have the identities 
\begin{align}
\Phi_{N,t}=\wt{\Phi}_{N,t} \circ \pi_{\leq N} + S(t)\circ \pi_{>N} \quad \text{and}\quad \pi_{\leq N}\circ \Phi_{N,t}=\wt{\Phi}_{N,t}\circ \pi_{\leq N}. \label{flows}
\end{align}
In this section, we define $E_N$  by
\begin{align*}
E_N=\pi_{\leq N}H^{\s}(\T)=\text{span} \{e^{inx}:\vert n \vert \leq N\}
\end{align*}

\noi
and let $E_N^\perp$ be the orthogonal complement of $E_N$
in $H^{\s}(\T)$. Given $R > 0$, we use $B_R$ to denote the ball of radius $R$ in $H^\s(\T)$
centered at the origin.
We need some boundedness and approximation properties for these flows.

\begin{lemma}\label{LEM:approx}
Given $\frac{1}{2}<\al<1$, let $s>  1-\frac{\al}{2}$ be such 
that the flow $\Phi$ of FNLS~\eqref{FNLS} is (locally) well-defined in $H^\s(\T)$, $\s=s-\frac 12-\eps$. 
Then, the following statements hold: 
\begin{enumerate}[\normalfont (i)]
\setlength\itemsep{0.3em}
\item  Then, for every $R>0$, there exist $T(R)>0$ and $C(R) >0$ such that 
\begin{align*}
\Phi_{N,t}\left(B_R \right)\subset B_{C(R)}
\end{align*}
for all $t\in[-T(R),T(R)]$ and for all $N\in \mathbb{N}\cup\{\infty\}$. 
\item For every $R>0$, there exists $T(R)>0$ and $C(R)>0$ such that 
\begin{align*}
\sup_{u\in B_R}  \|\Phi_{N,t}(u)\|_{X^{\s,\frac{1}{2}+}_{T(R)}}\leq C(R),
\end{align*}
uniformly in $N\in \NB \cup \{\infty\}$
\item Let $A\subset B_{R} \subset H^{\s}(\T)$ be a 
compact set and denote by $T(R)>0$ the local existence time 
of the solution map $\Phi$ defined on $B_{R}$. Then, for every $\delta>0$,
 there exists $N_0\in\mathbb{N}$, such that 
\begin{align*}
\|\Phi_{t}(u)-\Phi_{N,t}(u)\|_{H^{\s}}<\dl,
\end{align*} for any $u\in A$, $N\geq N_0$ and $t\in [-T(R),T(R)]$. Furthermore, we have
\begin{align*}
\Phi_{t}\left(A  \right)\subset \Phi_{N,t}\left(A+B_\delta\right)
\end{align*}
for all $t\in[-T(R),T(R)]$ and for all $N\geq N_0$.
\end{enumerate}

\end{lemma}

\begin{proof}
The uniform growth bounds (i) and (ii) follow from the local well-posedness argument in \cite{Cho}. 
The approximation property (iii) follows from a modification of the contraction argument as in \cite[Lemma 6.20]{OTz} by using the nonlinear estimate (the trilinear estimate) in \cite[Proposition 4.1]{Cho} and the following uniform estimate which plays the same role as in \cite[Lemma 6.19]{OTz}: there exists $N_0=N_0(\eps,R)\in \NB$ such that we have 
\begin{align}
\| \P_{>N} \Phi_{t}(u_0) \|_{X_{T(R)}^{\s, \frac 12+ }} < \eps 
\label{uni1}
\end{align}
for all $u_0 \in A$ and $N \ge N_0$. This estimate \eqref{uni1} can be proved by following the proof of \cite[Lemma 6.19]{OTz} with the continuity of the flow map $\Phi_{t}$ from $H^\s$ to $X_{T(R)}^{\s,\frac 12+}$. 
\end{proof}

\subsection{Proof of Theorem \ref{THM:2}}
In this subsection, we present the proof of our second main result (Theorem \ref{THM:2}). 
First, suppose that we have the following energy estimate with smoothing which will be proved in the next subsection.

\begin{lemma}\label{LEM:energy2} 
Let $\frac{1}{2}<\al<1$ and $s\in \big( \frac{3-2\al}{2},1\big]$. Given $N\in \NB $, $R>0$ and $u_0 \in B_{R}$, let $u_{N}\in C([0,T(R));H^{\s}(\T))$ be the local-in-time solution to \eqref{TFNLS21} satisfying $u_{N}\vert_{t=0}=\pi_{\leq N}u_0$, as assured by Lemma~\ref{LEM:approx}. Then there exists $k\in \NB$ such that, we have 
\begin{align}
\begin{split}
\bigg| \int_{0}^{T} \Re \jb{ i|u_N|^{2}u_N(-t,\cdot),D^{2s}u_N(-t,\cdot)}_{L^2(\T)}dt \bigg|  
&\les 1+\| u_N\|_{X_{T}^{\s,b}}^{k} \leq C(R).
\end{split}
\label{energyest2}
\end{align}

\noi
for some $b=\frac{1}{2}+\dl$, $0<\dl\ll 1$.
\end{lemma}
\noi
We are now ready to present the proof of Theorem \ref{THM:2}.

\begin{proof}[Proof of Theorem~\ref{THM:2}]
%
\noi
Following the same argument as in the proof of Proposition~\ref{PROP:den}, we obtain an explicit expression for the Radon-Nikodym derivative:
\begin{align*}
f_N(t,\phi_N)&=\frac{d({\wt \Phi}_{N,t} )_*\mu_{s,N} }{d\mu_{s,N}}(\phi_N)\\
&=\exp \bigg(\mp\int_0^t \Re \jb{i(|u_N|^2u_N )(-s,\phi_N) , D^{2s}u_N(-s,\phi_N) }_{L^2_{x}}  ds \bigg).
\end{align*}
Now we want to show (local-in-time and local-in-phase space) quasi-invariance. 
Given $R>0$, let $A\subset B_{R} \subset H^{\s}(\T)$ be a measurable set. We want to show that
given any $t\in [0,T(R)]$, we have
\begin{align*}
\mu_s(A)=0 \implies \mu_s(\Phi_{t}(A))=0.
\end{align*}
By the inner regularity of the measure $\mu_s$, it is enough to show that
\begin{align*}
A\subset B_{R} \textnormal{ compact and }\mu_s(A)=0 \implies \mu_s(\Phi_{t}(A))=0.
\end{align*}
Now, from Lemma~\ref{LEM:approx} (ii), we have 
\begin{align}
\mu_{s}(\Phi_{t}(A)) \leq \mu_{s}( \Phi_{N,t}(A+B_{\dl})) \label{meastoN}
\end{align}
for any fixed $\dl>0$, provided that $N$ is large enough. Let $D\subset B_{2R}$ be an arbitrary measurable set. By Fubini's theorem and \eqref{flows}, we have 
\begin{align*}
\mu_{s}(\Phi_{N,t}(D)) & = \int \chi_{D}( \Phi_{N,-t}( \phi) )d\mu_{s}(\phi) \\
& = \int_{E_{N}^{\perp}} \bigg\{ \int_{E_{N}} \chi_{D}( \Phi_{N,-t}(\phi)) d\mu_{s,N} \bigg\}  d\mu_{s,N}^{\perp} \\
& =\int_{E_{N}^{\perp}} \bigg\{ \int_{E_{N}} \chi_{D}(  \wt{\Phi}_{N,-t}(\pi_{\leq N} \phi)+S(-t)\pi_{>N}\phi)  d\mu_{s,N} \bigg\}  d\mu_{s,N}^{\perp}  \\
& =\int_{E_{N}^{\perp}} \bigg\{ \int_{E_{N}} \chi_{D}(  \phi_{N}+S(-t)\pi_{>N}\phi) f_{N}(-t,\phi_{N})  d\mu_{s,N} \bigg\}  d\mu_{s,N}^{\perp} 
\end{align*} 
Since $D\subset B_{2R}$,  we can use Lemma~\ref{LEM:approx} and Lemma~\ref{LEM:energy2} to find 
\begin{align*}
\sup_{N\in \NB} \sup_{\phi_{N} \in B_{2R}}\sup_{t\in [ -T(R),T(R)]}  f_{N}(-t,\phi_N) \leq e^{C_{0}(T(R),R)} =: C(R).
\end{align*}
By the invariance of $\mu_{s,N}^{\perp}$ under $S(t)$, which follows from \cite[Lemma 4.1]{OTz}, we then have
\begin{align*}
\mu_{s}(\Phi_{N,t}(D)) & \leq C(R) \int_{E_{N}^{\perp}} \int_{E_{N}} \chi_{D}(  \phi_{N}+S(-t)\pi_{>N}\phi)   d\mu_{s,N}d\mu_{s,N}^{\perp}  \\
& = C(R) \int_{E_{N}} \int_{E_{N}^{\perp}} \chi_{D}(  \phi_{N}+S(-t)\pi_{>N}\phi)  d\mu_{s,N}^{\perp}  d\mu_{s,N} \\
&= C(R) \int_{E_{N}} \int_{E_{N}^{\perp}} \chi_{D}(  \phi_{N}+\pi_{>N}\phi)  d\mu_{s,N}^{\perp}  d\mu_{s,N} \\
& = C(R) \int \chi_{D}( \phi )  d\mu_{s} = C(R)\mu_{s}(D).
\end{align*}
We now go back and apply this inequality with $D=A+B_{\dl}$ to \eqref{meastoN}  (supposing $\dl<R$), and we get
\begin{align*}
\mu_{s}(\Phi_{t}(A))  \leq C(R) \mu_{s}(A+B_{\dl})
\end{align*}
Since $A$ is compact, by the continuity of probability measures from above, we have 
\begin{align*}
\lim_{\dl\to 0}  \mu_{s}( A+B_{\dl}) =  \mu_{s}(A)
\end{align*}

\noi
and hence 
\begin{align}
\mu_s(\Phi_{t}(A) ) \le C(R) \mu_s(A)
\label{Com}
\end{align}

\noi
for any compact set $A \subset B_R$.
Now since $\mu_{s}(A)=0$, we have  
\begin{align*}
\mu_{s}(\Phi_{t}(A))=0.
\end{align*}

\noi
This completes the proof of Theorem \ref{THM:2}.
\end{proof}

We note that we can remove the compactness assumption in \eqref{Com}. This general observation is crucial for the proof of Proposition~\ref{PROP:Linfty} below.

\begin{corollary}\label{COR:density}
Let $\frac{1}{2}<\al<1$ and $s\in \big( \frac{3-2\al}{2},1\big]$. Given any $R>0$, let $A\subseteq B_{R}$ be a measurable set. Then, 
\begin{align*}
\mu_{s}(\Phi_{t}(A))\leq C(R)\mu_{s}(A),
\end{align*}
for any $t\in [0,T(R)]$.
\end{corollary}

\begin{proof}
We follow the argument in \cite[Lemma 6.10]{OTz}. 	
Let $A$ be a measurable set in $B_R \subset H^\s$. Then, from the inner regularity of $\mu_s$, there exists a sequence $\{K_j \}$ of compact sets such that $K_j \subset \Phi_{t}(A)$ and 
\begin{align}
\lim_{j\to \infty } \mu_s(K_j) =\mu_s(\Phi_{t}(A)).
\label{Com1}
\end{align}  

\noi
From the bijectivity of $\Phi_{t}$, we have
\begin{align*}
K_j=\Phi_{t}(\Phi_{-t}(K_j)).
\end{align*} 	

\noi
Since $\Phi(-t)$ is the continuous map, one can observe that $\Phi(-t)(K_j)$ is compact. Also, we have $\Phi(-t)(K_j)\subset \Phi(-t)\Phi_{t}(A)=A$.
Hence, by applying \eqref{Com} to $\Phi(-t)(K_j)$, we have
\begin{align}
\mu_s(K_j)=\mu_s(\Phi_{t}( \Phi(-t)(K_j) ) )
\le C(R) \mu_s(\Phi(-t)(K_j))
\le C(R) \mu_s(A).
\label{Com2}
\end{align} 	 
	
\noi
Then, after taking a limit as $j\to \infty$, it follows from \eqref{Com2} and \eqref{Com1} that we have the desired result.  	
\end{proof}

\begin{proof}[Proof of Proposition~\ref{PROP:Linfty}]
In the following, we fix $t\in \R$.
We first show \eqref{density}. 
By Proposition~\ref{PROP:den}, we know that the density of $(\Phi_{t})_{\ast} \mu_{s,r}$ with respect to $\mu_{s,r}$ is given by
\begin{align}
\frac {d({\Phi_t})_*\mu_{s,r} }{d\mu_{s,r} }=f(t,\cdot)=
\exp\bigg(\mp\int_0^t   \textup{Re} \, \jb{i (|u|^2u)(-t',\cdot ), D^{2s}u(-t',\cdot)   }_{L^2(\T)}    dt' \bigg)
\label{den}
\end{align}
for every $t\in \R$ and $r>0$. Now, fix $t\in \R$,  $r>0$ and $B_{r}:=\{ \phi \in L^2(\T)\,:\, \|\phi\|_{L^2}\leq r\}$.
Let $A\subseteq L^2(\T)$ be a measurable set. 
From the $L^2$-conservation of the flow of \eqref{FNLS} and \eqref{den}, we have 
\begin{align}
(\Phi_{t})_{\ast}\mu_s (A\cap B_r)&= \int_{L^2(\T)} \chi_{A\cap B_r}(  \Phi_{t}(\phi))  d\mu_s(\phi) \notag \\
& = \int_{L^2(\T)} \chi_{A}(\Phi_{t}(\phi))\chi_{B_r}(\phi)d\mu_s (\phi)  \notag\\
& = \int_{L^2(\T)} \chi_{A}(\phi) d(\Phi_{t})_{\ast}\mu_{s,r}(\phi) \notag \\
& =\int_{L^2(\T)} \chi_{A}(\phi) \chi_{B_r}(\phi)f(t,\phi)d\mu_{s}(\phi). \label{meas2}
\end{align} 

\noi
It follows from the continuity from below of a measure that 
\begin{align}
\lim_{r\to \infty }(\Phi_t)_{*}\mu_s(A\cap B_r)=(\Phi_t)_{*}\mu_s(A).
\label{cont}
\end{align}

\noi
From the Lebesgue monotone convergence theorem, we have
\begin{align}
\lim_{r\to \infty}\int_A \chi_{B_r}(\phi)f(t,\phi)d\mu_{s}(\phi)=\int_A f(t,\phi)d\mu_{s}(\phi).
\label{mono}
\end{align} 

\noi
Hence, by combining \eqref{meas2}, \eqref{cont}, and \eqref{mono}, we have
\begin{align}
(\Phi_t)_{*}\mu_s(A)=\int_A f(t,\phi)d\mu_{s}(\phi)
\label{den2}
\end{align}

\noi
for any measurable set $A\subseteq L^2(\T)$. It follows from the definition of the Radon–Nikodym derivative $\frac {d({\Phi_t})_*\mu_s }{d\mu_s }$ (that is, 
$\frac {d({\Phi_t})_*\mu_s }{d\mu_s }$ is a function (up to a $\mu_s$-null set) which satisfies  
\eqref{den2}) that we have
\begin{align*}
\frac {d({\Phi_t})_*\mu_s }{d\mu_s }=f(t,\cdot)=
\exp\bigg(\mp\int_0^t   \textup{Re} \, \jb{i (|u|^2u)(-t',\cdot ), D^{2s}u(-t',\cdot)   }_{L^2(\T)}    dt' \bigg)
\end{align*}

\noi
and hence we have verified \eqref{density}.

Now, we proceed to show that $f(t,\cdot)\in L^{\infty}_{\text{loc}}(d\mu_s)$. By following the proof of Theorem~\ref{THM:2} but replacing each instance of Lemma~\ref{LEM:approx} and Lemma~\ref{LEM:energy2} by Lemma~\ref{LEM:Superapprox} and Lemma~\ref{LEM:energy}, we obtain:
\begin{align}
\mu_s(\Phi_{t}(A)) \leq C(R)\mu_s (A), \label{measbd}
\end{align}
for any compact $A\subseteq B_{R}\subset L^2(\T)$ and for some $C(R)>0$ independent of $A$. Following the argument in the proof of Corollary~\ref{COR:density}, we obtain \eqref{measbd} for any measurable set $A\subseteq B_{R}\subset L^2(\T)$. Fix $t\in \R$ and $R>0$. We claim that $\chi_{B_R} f(t,\cdot)\in L^{\infty}(d\mu_s)$. We argue as in \cite[Proposition 3.5]{BTh}.
Suppose, in order to obtain a contradiction, that $\chi_{B_R} f(t,\cdot)\notin L^{\infty}(d\mu_s)$. Then, choosing $M=2C(R)$, there exists a measurable set $A_{R}$ such that 
\begin{align*}
\chi_{B_R}(\phi)f(t,\phi)>M \quad \text{for all} \quad \phi\in A_R
\end{align*}
and $\mu_{s}(A_R)>0$. We necessarily have that $A_{R} \subseteq B_{R}$. Now
\begin{align*}
\mu_{s}(\Phi_{t}(A_R))&=\int_{A_R} f(t,\phi)d\mu_s (\phi)=\int_{A_R} \chi_{B_{R}}(\phi)f(t,\phi)d\mu_s (\phi)> M\mu_{s}(A_R).
\end{align*}

\noi
Now our choice of $M$ yields a contradiction with \eqref{measbd}. Hence, $\chi_{B_R} f(t,\cdot)\in L^{\infty}(d\mu_s)$ and since $R>0$ was arbitrary, $f(t,\cdot)\in L^{\infty}_{\text{loc}}(d\mu_s)$.
\end{proof}

\subsection{Proof of Lemma~\ref{LEM:energy2}}

In this subsection, we present the proof of Lemma \ref{LEM:energy2}. 
We first proceed exactly as in the proof of Lemma~\ref{LEM:energy} by observing a cancellation of resonant interactions, symmetrizing, and integrating by parts in time. This reduces the proof of \eqref{energyest2} to establishing the following two estimates:
\begin{align}
\sup_{t\in [0,T(R)]} \big| \N_{0}(u_N)(t)\big| & \les \sup_{t\in [0,T(R)]}\|u_{N}(t)\|_{H^{\s}}^{4}, \label{bdry1} \\
  \sum_{j=1}^{2} \big|  \N_{j}(u_N) \big|   & \les 1+\|u_{N}\|_{X_{T(R)}^{\s,b}}^{6}, \label{higher1}
\end{align}
for any $N\in \NB\cup\{\infty\}$ and some $b=\frac{1}{2}+\dl$, $0<\dl \ll1$, where the multilinear operators $\N_{0}, \N_{1}$ and $\N_{2}$ are defined in \eqref{intebypart}, with $|n_1|\geq |n_3|$ and $|n_4|\geq |n_2|$.
As our estimates will be uniform in the parameter $N$, in the following, we will simply write $u$ for $u_N$.
First, we will show 
\begin{align}
&|\N_{1}(u_N)|  \notag \\
&= \bigg\vert 2\text{Re} \int_{-T}^{0}\sum_{n_4\in \Z}\sum_{  \substack{n_4=n_1-n_2+n_3 \\ n_1,n_3\neq n_4}  }\frac{\Psi_{s}(\cj{n})}{\Phi(\cj{n})}  \bigg(\sum_{n_1=n_{11}-n_{12}+n_{13}} \ft u(t,n_{11})\cj{ \ft u(t,n_{12})} \ft u(t,n_{13})  \bigg) \notag \\
&\hphantom{XXX}\times \cj{ \ft u}(t,n_2)\ft u(t,n_3) \cj{ \ft u}(t,n_4) dt \bigg\vert \notag \\
& \les \| u\|_{X_{T}^{\s,b}}^{6},  \label{Remainderest2}
\end{align}
for some $b=\frac{1}{2}+\dl$, $0<\dl\ll 1$.
The same argument can be applied to prove \eqref{Remainderest2} for $\N_{2}(u_N)$. Thus, we neglect to show an estimate for $\N_{2}(u_N)$.
In the following, we let $w$ be any extension of $u$ on $[-T,T]$. We suppose that $|n_{11}|\geq |n_{13}|$. Let $n_{\text{max}}^{(1)}=\max( |n_{11}|,|n_{12}|)+1$. 
In this setting, it will be convenient to dyadically decompose the frequencies of all the functions. In view of our symmetry assumptions, we therefore have $N_{1}\geq N_3$, $N_{4}\geq N_2$ and $N_{11}\geq N_{13}$.
We let 
\begin{align*}
f_{N}:&=  \F_{t,x}^{-1}( | \jb{n}^\s \mathcal{F}\{\P_{N}w\}(\tau,n) | ),\\
f_{N,T}:&= \F_{t,x}^{-1}( | \jb{n}^\s \mathcal{F}\{\chi_{[-T,0]}  \P_{N}w \}(\tau,n)| ).
\end{align*}

We split the frequency region into a few cases in a similar way as we did in the proof of Lemma~\ref{LEM:energy}.

\medskip
\noi
$\bullet$ \underline{ \textbf{Case 1:} $n^{(1)}_{\text{max}} \gg n_{\text{max}}$} 

\medskip
\noi

In this case, we have $n^{(1)}_{\text{max}} \gg |n_1|$ which implies $|n_{11}|\sim |n_{12}|$ or $|n_{11}|\sim |n_{13}|$. We assume $|n_{11}|\sim |n_{12}|$.
Then, we have 
\begin{align}
|\N_{1}(u_N)|& \leq \sum_{  \substack{   \scriptscriptstyle{ N_j, \, j=1,\ldots,4}  \\ \scriptscriptstyle{N_{1k}, k=1,2,3}\\ \scriptscriptstyle{N_{1}\geq N_3, N_{4}\geq N_2} \\ \scriptscriptstyle{N_{11}\geq N_{13}, N_{11}\sim N_{12}}  }} (N_2N_3N_4 N_{11}N_{12}N_{13} )^{-\s} \intt_{  \scriptscriptstyle{ \tau_1-\tau_2+\tau_3-\tau_4 +\tau_5-\tau_6=0}} \sum_{\substack{ \scriptscriptstyle{  n_1-n_2+n_3-n_4=0} \\ \scriptscriptstyle{n_1=n_{11}-n_{12}+n_{13}}}}  \frac{|\Psi_{s}(\cj{n})|}{|\Phi_{\al}(\cj{n})|}   
\notag\\
& \times \prod_{j=2}^{4} \mathcal{F}_{t,x} f_{N_j}(\tau_j, n_j)   \mathcal{F}_{t,x} f_{N_{11}, T}(\tau_4, n_{11})   \mathcal{F}_{t,x} f_{N_{12}}(\tau_5, n_{12})\mathcal{F}_{t,x} f_{N_{13}}(\tau_6, n_{13}) d\tau_{1} \ldots d\tau_{5} \notag \\
& =: \sum_{  \ast }  I(\cj{N}), \notag
\end{align}
where $\ast$ represents the conditions on the summations in the first line above.
We split into a few subcases.

\medskip
\noi
$\bullet$ \underline{ \textbf{Case 1.1:} $|n_{13}|\sim |n_{11}|$} 

\medskip
\noi

Without loss of generality, we assume $|n_4|\sim n_{\text{max}}$. 
From Lemma~\ref{LEM:mult}, we have 
\begin{align*}
	\frac{ | \Psi_{s}(\cj{n})|}{|\Phi_{\al}(\cj{n})|} \les N_{4}^{2s-2\al}. 
	\end{align*}
Hence, by H\"{o}lder's inequality, \eqref{bilinstrich1}, \eqref{L4strich}, \eqref{Linfty} and Lemma~\ref{LEM:sharp}, we have 
\begin{align*}
I(\cj{N}) & \les N_{4}^{2s-2\al-\s}N_{11}^{-3\s}(N_2 N_3)^{-\s} \big\| f_{N_{11},T}f_{N_4}\|_{L^2_{t,x}} \|f_{N_{12}}\|_{L^4_{t,x}} \|f_{N_{13}}\|_{L^4_{t,x}} \|f_{N_{2}}\|_{L^{\infty}_{t,x}}\|f_{N_3}\|_{L^{\infty}_{t,x}} \\
& \les N_{4}^{2s-2\al-\s+\frac{1-\al}{2}}N_{11}^{-3\s+\frac{1-\al}{2}} (N_2 N_3)^{-\s+\frac{1}{2}} \prod_{j=2}^{3}\|f_{N_j}\|_{X^{0,\frac{1}{2}+}} \|f_{N_4}\|_{X^{0,\frac{3}{8}}} \| f_{N_{12}}\|_{X^{0,\frac{3}{8}}}\\
 & \hphantom{XXXXXXXXXXXXXXXXXXXXXX} \times\| f_{N_{13}}\|_{X^{0,\frac{3}{8}}} \|f_{N_{11},T}\|_{X^{0,\frac{3}{8}}} \\
 & \les N_{4}^{2s-2\al-3\s+\frac{1-\al}{2}+1}N_{11}^{-3\s+\frac{1-\al}{2}} \prod_{j=2}^{4}\|\P_{N_j}w\|_{X^{\s,\frac{1}{2}+}} \prod_{\l=1}^{3}\|\P_{N_{1 \l }} w\|_{X^{\s,\frac{1}{2}+}}.
\end{align*}
To sum over the dyadic scales, we enforce $-3\s+\frac{1-\al}{2} \leq -3\eps<0$, which requires $s\geq \frac{4-\al}{6}+2\eps$. Then, we can write 
\begin{align*}
N_{4}^{2s-2\al-3\s+\frac{1-\al}{2}+1}N_{11}^{-3\s+\frac{1-\al}{2}} \les N_{11}^{-3\eps}N_{4}^{2s-3\al-6\s+2+3\eps} \les N_{11}^{-3\eps},
\end{align*}
provided that $2s-3\al-6\s+2+3\eps\leq 0$. This conditions requires $s> \frac{5-3\al}{4}+\frac{3}{2}\eps$. Hence, in this case, we need 
\begin{align}
s> \max\bigg(  \frac{4-\al}{6}+2\eps, \frac{5-3\al}{4}+\frac{3}{2}\eps\bigg)=\frac{5-3\al}{4}+\frac{3}{2}\eps, 
\label{salcond1}
\end{align} 
for $0<\eps\ll 1$. Thus, we may sum over the dyadic scales and show that this contribution can be bounded by the right hand side of \eqref{Remainderest2}.

\medskip
\noi
$\bullet$ \underline{ \textbf{Case 1.2:} $|n_{11}| \gg |n_{13}|   \gg n_{\text{max}}$} 

\medskip
\noi

We proceed similar to Case 1.1, with  
\begin{align*}
I(\cj{N}) & \les N_{4}^{2s-2\al-\s} N_{13}^{-\s} N_{11}^{-2\s}  (N_2 N_{3})^{-\s}\big\| f_{N_{11},T}f_{N_4}\|_{L^2_{t,x}} \|f_{N_{12}}\|_{L^4_{t,x}} \|f_{N_{13}}\|_{L^4_{t,x}} \|f_{N_{2}}\|_{L^{\infty}_{t,x}}\|f_{N_3}\|_{L^{\infty}_{t,x}} \\
&\les N_{4}^{2s-2\al-\s +\frac{1-\al}{2} }N_{13}^{-\s+\frac{1-\al}{4}} N_{11}^{-2\s+\frac{1-\al}{4}} (N_2 N_3)^{-\s+\frac{1}{2}}  \prod_{j=2}^{4}\|\P_{N_j}w\|_{X^{\s,\frac{1}{2}+}} \prod_{\l=1}^{3}\|\P_{N_{1 \l }} w\|_{X^{\s,\frac{1}{2}+}} \\
& \les N_{4}^{2s-3\al -6\s +2+} N_{11}^{-3\eps} \prod_{j=2}^{4}\|\P_{N_j}w\|_{X^{\s,\frac{1}{2}+}} \prod_{\l=1}^{3}\|\P_{N_{1 \l }} w\|_{X^{\s,\frac{1}{2}+}} \\
& \les N_{11}^{-3\eps} \prod_{j=2}^{4}\|\P_{N_j}w\|_{X^{\s,\frac{1}{2}+}} \prod_{\l=1}^{3}\|\P_{N_{1 \l }} w\|_{X^{\s,\frac{1}{2}+}}
\end{align*}
provided $s$, $\al$ and $\eps$ satisfy \eqref{salcond1}.

\medskip
\noi
$\bullet$ \underline{ \textbf{Case 1.3:} $ |n_{13}|   \ll n_{\text{max}}$} 

\medskip
\noi

We may assume $|n_{\text{max}}|\sim |n_4|$. Then,
in this case, we proceed in a similar way as in Case 1.2 above.

\medskip
\noi
$\bullet$ \underline{ \textbf{Case 2:} $ n_{\text{max}}^{(1)}  \sim n_{\text{max}}$} 

\medskip
\noi

\medskip
\noi
$\bullet$ \underline{ \textbf{Case 2.1:} $  n_{\text{max}}  \gg |n_1|$} 

\medskip
\noi

In this case, $|n_\text{max}|\sim |n_4|\sim |n_2|$. Hence, we can proceed as in Case 1.1.

\medskip
\noi
$\bullet$ \underline{\textbf{Case 2.2:} $ |n_{\text{max}}| \sim |n_1|$}

\medskip
\noi
We can have $|n_1|\sim |n_3|$ or $|n_1|\sim |n_4|$. We assume $|n_1|\sim |n_3|$. This leads to two natural subcases. 

\medskip
\noi
$\bullet$ \underline { \textbf{Subcase 2.2.1:} $|n_{11}| \gg   |n_{12}| $}

\medskip
\noi

In this case, we have $|n_{11}|\sim |n_{3}|$. We assume that 
$|n_4|=\max( |n_{12}|,|n_{13}|, |n_2|,|n_4|)$. If $|n_4|\sim n_{\text{max}}$, we may proceed as in Case 1.1. Otherwise, if $|n_4| \ll n_{\text{max}}$, we need to apply the bilinear estimate \eqref{bilinstrich1} twice. Indeed, suppose that $|n_2|=\max( |n_{12}|,|n_{13}|,|n_2|)$. Then, by H\"{o}lder's inequality, \eqref{bilinstrich1}, \eqref{Linfty} and Lemma~\ref{LEM:sharp}, we have  
\begin{align*}
I(\cj{N})& \les N_{3}^{2s-2\al-2\s} N_{4}^{-\s}(N_{12}N_{13}N_{2})^{-\s} \|f_{N_{11},T}f_{N_4 }\|_{L^2_{t,x}}\|f_{N_3}f_{N_2}\|_{L^2_{t,x}} \|f_{N_{12}}\|_{L^{\infty}_{t,x}}\|f_{N_{13}}\|_{L^{\infty}_{t,x}}\\
& \les N_{3}^{2s-2\al-2\s} (N_{4}N_{2})^{-\s+\frac{1-\al}{4}} (N_{12}N_{13})^{-\s+\frac{1}{2}} \prod_{j=2}^{4}\|\P_{N_j}w\|_{X^{\s,\frac{1}{2}+}} \prod_{\l=1}^{3}\|\P_{N_{1 \l }} w\|_{X^{\s,\frac{1}{2}+}} \\
& \les N_{3}^{-2\eps}N_{4}^{2s-2\al-6\s+\frac{1-\al}{2}+1} \prod_{j=2}^{4}\|\P_{N_j}w\|_{X^{\s,\frac{1}{2}+}} \prod_{\l=1}^{3}\|\P_{N_{1 \l }} w\|_{X^{\s,\frac{1}{2}+}} \\
& \les N_{3}^{-2\eps} \prod_{j=2}^{4}\|\P_{N_j}w\|_{X^{\s,\frac{1}{2}+}} \prod_{\l=1}^{3}\|\P_{N_{1 \l }} w\|_{X^{\s,\frac{1}{2}+}},
\end{align*} 
where in the third inequality we need $\al\geq \frac{1}{2}+2\eps$ and in the final inequality, we need 
\begin{align*}
s>\frac{9-5\al}{8}+\frac{3}{2}\eps.
\end{align*}

\noi
$\bullet$ \underline { \textbf{Subcase 2.2.2:} $|n_{11}|\sim  |n_{12}| \sim |n_1| $}

\medskip
\noi

In this case, we have $|n_{11}|\sim |n_{12}|\sim |n_3|$, so we can proceed as in Case 1.1.

\medskip
\noi
$\bullet$ \underline { \textbf{Case 3:} $n_{\text{max}}^{(1)}\ll n_{\text{max}}$}

Since $|n_1|\les n_{\text{max}}^{(1)}$, we have $n_{\text{max}}\sim |n|\sim |n_2|$. If $|n_3| \sim n_{\text{max}}$, we can proceed as in Case 1.1. Otherwise, if $|n_3|\ll n_{\text{max}}$, we may proceed as in Subcase 2.2.1.

Compiling these cases, overall we require $\frac{1}{2}<\al<1$ and $s>\frac{9-5\al}{8}+\frac{3}{2}\eps$. This completes the proof of \eqref{Remainderest2}.

\medskip
\noi

We now estimate the boundary term $\N_{0}$ and thus establish \eqref{bdry1}. We fix $t\in [0,T(R)]$. This leads to the following cases.

\medskip
\noi
$\bullet$ \underline{ \textbf{Case 1:} $|n_{(1)}|\sim |n_{(4)}|$} 

\medskip
\noi
From Lemma~\ref{LEM:mult}, we have 
\begin{align*}
	\frac{ | \Psi_{s}(\cj{n})|}{|\Phi_{\al}(\cj{n})|} \les |n_{(1)}|^{2s-2\al}. 
	\end{align*}
We define $f(t,n)=\jb{n}^{\s}  | \ft u (t, n)|$. Then, by Cauchy-Schwarz, we have 
\begin{align*}
|\N_0(u)(t)|& \les \sum_{n_4\in\Z}\sum_{  \substack{n_4=n_1-n_2+n_3 \\ n_1,n_3\neq n_4}  } \frac{1}{\jb{n_{(1)}}^{\nu}} f(n_1)f(n_2)f(n_3)f(n_4) \\
& \les \bigg(\sum_{n_4\in\Z}\sum_{  n_4=n_1-n_2+n_3 }  \frac{f(n_1)^{2} f(n_4)^2 }{\jb{n_3}^{\nu}} \bigg)^{\frac{1}{2}}\bigg(\sum_{n_4\in\Z}\sum_{  n_4=n_1-n_2+n_3 }  \frac{f(n_2)^{2} f(n_3)^{2}}{\jb{n_1}^{\nu}} \bigg)^{\frac{1}{2}} \\
&\les \|u(t)\|_{H^{\s}}^{4},
\end{align*}
where $\nu:=4\s+2\al-2s$ and we can perform the summations provided $s> \frac{3-2\al}{2}+\eps$.

\medskip
\noi
$\bullet$ \underline{ \textbf{Case 2:} $|n_{(1)}|\sim |n_{(3)}|\gg  |n_{(4)}|$} 

\medskip
\noi

We assume, without loss of generality, that $|n_1|\sim |n_2|\sim|n_3|\gg |n_4|$. Then, we have 

\begin{align*}
|\N_0(u)(t)|
& \les \bigg(\sum_{n_4\in\Z}\sum_{  n_4=n_1-n_2+n_3 }  f(n_1)^2 f(n_2)^{2} f(n_4)^2 \bigg)^{\frac{1}{2}}\bigg(\sum_{n_4\in\Z}\sum_{ \substack{ n_4=n_1-n_2+n_3 \\ |n_1|\ges |n_4|} }   \frac{f(n_3)^2}{\jb{n_1}^{2\nu} \jb{n_4}^{2\s}}   \bigg)^{\frac{1}{2}} \\
& \les \bigg( \sum_{ \substack{ n_1,n_4 \\|n_1|\ges |n_4| }} \frac{1}{\jb{n_1}^{1+\eps} \jb{n_4}^{1+\eps}}  \frac{1}{\jb{n_1}^{2\nu +2\s-2-2\eps}  }      \bigg)^{\frac{1}{2}} \|u(t)\|_{H^{\s}}^4  
\les \|u(t)\|_{H^{\s}}^{4}, 
\end{align*}
where $\nu:=3\s+2\al-2s$ and we can sum provided $2\nu +2\s-2-2\eps\geq 0$, which requires $s \ge \frac{3-2\al}{2}+5\eps$. We have also used here the condition $-2\s+1+\eps \ge 0 $, which requires $s \le 1+ \frac {3\eps}{2}$ but this condition is satisfied since we have only considered the case $s \le 1$.

\medskip
\noi
$\bullet$ \underline{ \textbf{Case 3:} $|n_{(1)}|\sim|n_{(2)}| \gg |n_{(3)}|\geq   |n_{(4)}|$} 

\medskip
\noi

With $\nu:=2\s+2\al-2s>0$, we have
\begin{align*}
\frac{1}{\jb{ n_{(1)}}^{\nu} \jb{n_{(3)}}^{\s} \jb{n_{(4)}}^{\s}  } \les \frac{1}{ \jb{n_{(3)}}^{\s+\frac{\nu}{2}} \jb{n_{(4)}}^{\s+\frac{\nu}{2}}  }.
\end{align*}
Hence, by Cauchy-Schwarz, 
\begin{align*}
|\N_0(u)(t)| & \les \bigg( \sum_{ n_{(3)}, n_{(4)} }   \frac{1}{ \jb{n_{(3)}}^{2\s+\nu} \jb{n_{(4)}}^{2\s+\nu  
}}   \bigg)^{\frac{1}{2}}  \|u(t)\|_{H^{\s}}^{4},
\end{align*}
where we can sum provided that $2\s+\nu>1$. This is satisfied if $s>\frac{3-2\al}{2}+2\eps$.

\begin{remark}\rm 
Combining the regularity restrictions required in estimating both the boundary and remainder piece above, we see that we need 
\begin{align*}
s>\max\bigg( \frac{9-5\al}{8}, \frac{3-2\al}{2}\bigg)=\frac{3-2\al}{2},
\end{align*}
since $\frac{1}{2}<\al<1$, which gives rise to the restriction in Lemma~\ref{LEM:energy2}. We point out that, contrary to the higher dispersion case (Lemma~\ref{LEM:energy}), the worst regularity restriction comes from the boundary estimate \eqref{bdry1}. This is an artefact of the weaker dispersion since the lower bound on the phase function in Lemma~\ref{lemma: phase lower bound} is much less effective and we can also no longer use space-time estimates like the $L^4_{t,x}$-Strichartz estimate \eqref{L4strich}.
\end{remark}

\begin{remark}\rm  \label{RMK:DTfrac}
Our reduction to the energy estimate in Lemma~\ref{LEM:energy2} is essential for studying the weakly dispersive FNLS~\eqref{FNLS}. If instead we tried to obtain an energy estimate as in Lemma~\ref{LEM:energy}, following \cite{DT2020}, we would need to place two functions into the $L^{\infty}_{t}L^{2}_{x}$-norm in order to be controlled using the $L_{x}^2$-conservation. To illustrate why this does not yield results for all $\frac{1}{2}<\al<1$, we consider the Subcase 2.2.1 in the above proof of Lemma~\ref{LEM:energy2}. Using the assumptions and notations from there along with Sobolev embedding and the $L^{4}_{t,x}$-Strichartz estimate \eqref{L4strich}, we have
\begin{align*}
I(\cj{N})& \les N_{3}^{2s-2\al}N_{3}^{-2\s}N_{4}^{-\s}N_{2}^{-\s} \|f_{N_3}\|_{L^{4}_{t}L^{\infty}_{x}}\|f_{N_{11}}\|_{L^{4}_{t}L^{\infty}_{x}}\|f_{N_{4}}\|_{L^{4}_{t}L^{\infty}_{x}}\|f_{N_{2}}\|_{L^{4}_{t}L^{\infty}_{x}} \\
& \hphantom{XXXXXXXXXXXXXXXXXXXX} \times \|u_{N_{12}}\|_{L^{\infty}_{t}L^{2}_{x}}\|u_{N_{13}}\|_{L^{\infty}_{t}L^{2}_{x}} \\
& \les N_{3}^{2s-2\al-2\s+\frac{1}{2}+\frac{1-\al}{2}+}N_{4}^{-2\s+\frac{1}{2}+\frac{1-\al}{2}+} \|f_{N_3}\|_{X^{0,\frac{3}{8}}}\|f_{N_{11}}\|_{X^{0,\frac{3}{8}}}\|f_{N_{4}}\|_{X^{0,\frac{3}{8}}}\|f_{N_{2}}\|_{X^{0,\frac{3}{8}}} \\
& \hphantom{XXXXXXXXXXXXXXXXXXXX} \times \|u_{N_{12}}\|_{L^{\infty}_{t}L^{2}_{x}}\|u_{N_{13}}\|_{L^{\infty}_{t}L^{2}_{x}}.
\end{align*}
In order to sum over the dyadic scales, we need 
\begin{align*}
2s-2\al-2\s+\frac{1}{2}+\frac{1-\al}{2}<0 \quad \text{and} \quad 2s-4\s-3\al+2<0.
\end{align*}
The first condition requires $\al>\frac{4}{5}$ and the second imposes $s>\frac{4-3\al}{2}$. Thus, such an approach does not seem to cover the full range $\frac{1}{2}<\al<1$.
\end{remark}

\subsection{Proof of Proposition~\ref{COR:stability}}

In this subsection, we prove Proposition~\ref{COR:stability}. The argument is the same as that of Corollary 1.3 in~\cite{ST1} and diverges only when we prove the stability in $L^{p}_{\text{loc}}(d\mu_s)$. 
We will include details for the benefit of the reader. 
We define the measures
\begin{align*}
d\nu_{j}(\phi) = g_{j}(\phi)d\mu_s(\phi), \quad j=1,2.
\end{align*}
Then, for any test function $\vp$, we have 
\begin{align*}
\int_{L^2} \vp(\phi)d(\Phi_{t})_{\ast}\nu_j (\phi)&= \int_{L^2} \vp(\Phi_{t}(\phi))d\nu_j (\phi)\\
&=  \int_{L^2} \vp(\Phi_{t}(\phi))g_{j}(\phi)d\mu_s (\phi) \\
& = \int_{L^2} \vp(\phi)g_{j}(\Phi_{-t}(\phi)) f(t,\phi)d\mu_s (\phi).
\end{align*}

\noi
Therefore, $d(\Phi_t)_{\ast}\nu_j(\phi)=G_{j}(t,\phi)d\mu_s (\phi)$, $j=1,2$, where $G_{j}(t,\phi)=g_{j}(\Phi_{-t}(\phi)) f(t,\phi)$. Now, we have 
\begin{align*}
\int_{L^2} |G_1(t,\phi)-G_2(t,\phi)|d\mu_s (\phi)& = \int_{L^2} |g_{1}(\Phi_{-t}(\phi)) -g_{2}(\Phi_{-t}(\phi)) | f(t,\phi)d\mu_s (\phi) \\
& = \int_{L^2} |g_1 (\phi)-g_2(\phi)|d\mu_s (\phi).
\end{align*}
Now, suppose $p>1$ and $g_1,g_2\in L^1(d\mu_s)\cap L^p_{\textup{loc}}(d\mu_s)$ and fix $R>0$. Then, by $L^2$-conservation, we have $\Phi_{-t}(B_R)=B_R$, and hence
\begin{align*}
&\int_{B_R}  |G_1(t,\phi)-G_2(t,\phi)|^p  d\mu_s (\phi)\\
& = \int_{B_R} |g_{1}(\Phi_{-t}(\phi)) -g_{2}(\Phi_{-t}(\phi)) |^p f(t,\phi)^{p}d\mu_s (\phi) \\
& \leq \| \chi_{B_R} f(t,\cdot)\|_{L^{\infty}(d\mu_s)}^{p-1} \int_{B_R} |g_{1}(\Phi_{-t}(\phi)) -g_{2}(\Phi_{-t}(\phi)) |^p f(t,\phi)d\mu_s (\phi) \\
& \leq \| \chi_{B_R} f(t,\cdot)\|_{L^{\infty}(d\mu_s)}^{p-1} \int_{\Phi_{-t}(B_R)} |g_1 (\phi)-g_2(\phi)|^p d\mu_s (\phi) \\
& = \| \chi_{B_R} f(t,\cdot)\|_{L^{\infty}(d\mu_s)}^{p-1} \int_{B_{R}} |g_1 (\phi)-g_2(\phi)|^p d\mu_s (\phi).
\end{align*}
This shows \eqref{Lpstab}.

\begin{ackno}\rm
The authors would like to kindly thank Tadahiro Oh and Nikolay Tzvetkov for suggesting the problem, for their continued support and for informing us that the $L^2(d\mu_s)$-integrability assumption in~\cite[Corollary 1.4]{OTz2} can be weakened to $L^1(d\mu_s)$-integrability. The authors are also grateful to Nikolay Tzvetkov for suggesting the application of Proposition~\ref{PROP:Linfty} to the $L^p$-stability result in Proposition~\ref{COR:stability}. K.S.~would like to express his gratitude to the School of Mathematics at the University of Edinburgh for its 
hospitality during his visit, 
where this manuscript was prepared. 
The authors also wish to thank the anonymous referees for their helpful comments. 

J.\,F.~was supported by The Maxwell Institute Graduate School in Analysis and its
Applications, a Centre for Doctoral Training funded by the UK Engineering and Physical
Sciences Research Council (grant EP/L016508/01), the Scottish Funding Council, Heriot-Watt
University and the University of Edinburgh and Tadahiro Oh's ERC starting grant no. 637995 “ProbDynDispEq”. K.S.~was partially supported by National Research Foundation of
Korea (grant NRF-2019R1A5A1028324).

\end{ackno}

\end{document}